\documentclass[10pt, english]{amsart}

\usepackage{amsmath,amssymb,enumerate}

\usepackage[T1]{fontenc}
\usepackage[all]{xy}

\usepackage{mathtools}
\usepackage{babel}
\usepackage{amstext}
\usepackage{amsmath}
\usepackage{amsfonts}
\usepackage{latexsym}
\usepackage{ifthen}

\usepackage[table]{xcolor}
\usepackage{longtable}
\newcolumntype{M}[1]{>{\centering\arraybackslash}m{#1}}

\usepackage{xypic}
\xyoption{all}
\newcommand{\codim}{{\rm codim}}

\newcommand{\sK}{{\mathcal K}}

\newtheorem{lemma1}{}[section]

\newenvironment{lemma}{\begin{lemma1}{\bf Lemma.}}{\end{lemma1}}

\newenvironment{theorem}{\begin{lemma1}{\bf Theorem.}}{\end{lemma1}}
\newenvironment{proposition}{\begin{lemma1}{\bf Proposition.}}{\end{lemma1}}
\newenvironment{corollary}{\begin{lemma1}{\bf Corollary.}}{\end{lemma1}}
\newenvironment{remark}{\begin{lemma1}{\bf Remark.}\rm}{\end{lemma1}}

\newenvironment{definition}{\begin{lemma1}{\bf Definition.}}{\end{lemma1}}

\newenvironment{remark*}{{\bf Remark.}}{}
\newenvironment{example*}{{\bf Example.}}{}
\newenvironment{assumption*}{{\bf Assumption.}}{}

\newcommand{\R}{\ensuremath{\mathbb{R}}}
\newcommand{\Q}{\ensuremath{\mathbb{Q}}}
\newcommand{\Z}{\ensuremath{\mathbb{Z}}}
\newcommand{\C}{\ensuremath{\mathbb{C}}}
\newcommand{\N}{\ensuremath{\mathbb{N}}}
\newcommand{\PP}{\ensuremath{\mathbb{P}}}

\usepackage{eurosym}

\newcommand{\holom}[3]{\ensuremath{#1:#2  \rightarrow #3}}
\newcommand{\fibre}[2]{\ensuremath{#1^{-1} (#2)}}

\makeatletter
\ifnum\@ptsize=0  \fi \ifnum\@ptsize=2  \fi 

\newcommand\sG{{\mathcal G}}

\newcommand\sI{{\mathcal I}}
\newcommand\sT{{\mathcal T}}
\newcommand\sU{{\mathcal U}}

\newcommand\sO{{\mathcal O}}

\newcommand\sV{{\mathcal V}}

\DeclareMathOperator*{\Sym}{Sym}
\DeclareMathOperator*{\pic}{Pic}
\DeclareMathOperator*{\sing}{sing}

\newcommand{\Chow}[1]{\ensuremath{\mbox{\rm Chow}(#1)}}

\usepackage{hyperref}

\DeclareMathOperator*{\Rat}{RatCurves^n}

\title{Examples of Fano manifolds with non-pseudoeffective
tangent bundle }
\date{\today}

\author{Andreas H\"oring}
\author{Jie Liu}
\author{Feng Shao}

\address{Andreas H\"oring, Universit\'e C\^ote d'Azur, CNRS, LJAD, France, Institut universitaire de France}
\email{Andreas.Hoering@unice.fr}

\address{Jie Liu, Morningside Center of Mathematics, Academy of Mathematics and Systems Science, Chinese Academy of Sciences, Beijing, 100190, China}
\email{jliu@amss.ac.cn}

\address{Feng Shao, Academy of Mathematics and Systems Science, Chinese Academy of Sciences, Beijing, 100190, China}
\email{shaofeng16@mails.ucas.ac.cn}


\begin{document}

\begin{abstract}
Let $X$ be a Fano manifold. While the properties of the anticanonical divisor
$-K_X$ and its multiples have been studied by many authors, the positivity
of the tangent bundle $T_X$ is much more elusive. We give a complete characterisation of the pseudoeffectivity of $T_X$ for del Pezzo surfaces, hypersurfaces in the projective space
and del Pezzo threefolds.
\end{abstract}

\newpage

\maketitle

\section{Introduction}

Let $X$ be a Fano manifold. While the properties of the anticanonical divisor
$-K_X$ and its multiples have been studied by many authors, the positivity
of the tangent bundle $T_X$ is much more elusive.
Peternell \cite[Thm.1.3]{Pet12} proved that $T_X$ is always generically ample, i.e. the restriction
$T_X|_C$ to a general complete intersection of sufficiently ample divisors is ample.
On the other hand it is clear that $T_X$ will not admit some positive metric
as this quickly leads to strong restrictions on the geometry of $X$,
see \cite[Thm.1.1]{Mat18}, \cite[Cor.1.3]{Iwa18} \cite[Thm.1.1]{HIM19}.
In this paper we will consider a property that is strictly weaker than the aforementioned metric properties:

\begin{definition}
    Let $X$ be a projective manifold. We say that the tangent bundle is pseudoeffective (resp. big) if the tautological class $c_1(\sO_{\PP(T_X)}(1))$ on the projectivised bundle
    $\PP(T_X)$ is pseudoeffective (resp. big).
\end{definition}

It has been shown by Hsiao \cite[Cor.1.3]{Hsi15} that the tangent bundle of a toric variety is big. In general it is difficult to give a numerical characterisation for bigness
of the tangent bundle,
even in low dimension: if $X$ is a del Pezzo surface, one has
$$
H^2(X, S^m T_X)= H^0(X, K_X \otimes S^m \Omega_X) = 0
$$
for all $m \in \N$ by Peternell's result, so by Riemann-Roch
$$
h^0(X, S^m T_X) \geq \chi(X, S^m T_X) = m^3 (c_1^2(X)-c_2(X)) + \sO(m^2).
$$
Yet since $c_1^2(X)+c_2(X)=12$ by Noether's formula, the condition $c_1^2(X)>c_2(X)$
is satisfied only for del Pezzo surfaces of degree at least seven, which are all toric.

In this paper we give pseudoeffectivity criteria for several families of Fano manifolds.
Our first main result settles the case of surfaces:

\begin{theorem}
	\label{THM:bigness-pseudo-effectivity-del-Pezzo-surfaces}
	Let $X$ be a smooth del Pezzo surface of degree $d$. Then the following holds.
	\begin{enumerate}[(a)]
		\item $T_X$ is pseudoeffective if and only if $d\geq 4$.
		\item $T_X$ is big if and only if $d\geq 5$.
	\end{enumerate}
	Moreover $T_X$ is pseudoeffective if and only if $H^0(X, S^m T_X) \neq 0$ for some $m \in \N$.
\end{theorem}

It is well-known that smooth del Pezzo surfaces of degree four and five do not admit non-zero vector fields, so the existence of symmetric tensors may be considered as a surprise. Since del Pezzo surfaces of degree five are not toric, Theorem \ref{THM:bigness-pseudo-effectivity-del-Pezzo-surfaces} gives, to our best knowledge,  the first example of a Fano manifold such that the tangent bundle is big, but $X$ is not almost homogeneous.
In particular our theorem answers questions of Hsiao \cite[Question 5.2]{Hsi15}  and Mallory \cite[Rem.5.9]{Mal20}. Let us also note that Mallory proved, with completely different techniques, that $T_X$ is not big if $d \leq 4$ \cite[Cor.5.7]{Mal20}.
Hosono, Iwai and Matsumura obtained similar results on pseudoeffectivity of the tangent bundle in the metric sense for rational surfaces  \cite[Prop.4.5, Prop.4.8]{HIM19}. For our proof we use the existence of conic bundle structures to construct divisors in $\PP(T_X)$ and compute their cohomology class.

We then turn to the case of higher-dimensional Fano hypersurfaces of the projective space. In this case 
the Picard number is one and the tangent bundle $T_X$ is stable which could be seen as an indicator
for a stronger connection between the positivity of $T_X$ and its determinant $-K_X$.
Moreover the theory of Schur functors can be used to compute explicitly
the space of global sections of certain twists of the symmetric powers of the tangent bundle.

Building on the work of Br\"uckmann and Rackwitz \cite{BrueckmannRackwitz1990}, we obtain partial generalisation of results of Schneider \cite{Schneider1992} and Bogomolov-de Oliveira \cite{BogomolovDeOliveira2008}:

\begin{theorem}
	\label{THM:Vanishing-twisted-symmetric-power-hypersurfaces}
	For $n\geq 2$, let $X\subset \mathbb{P}^{n+1}$ be a smooth hypersurface of degree $d\geq 3$. Then for any integer $k\geq 1$, we have
	\[
	H^0(X,{\rm Sym}^k (T_X\otimes \sO_{X}(d-3)))=0.
	\]
\end{theorem}

In particular we see that none of the symmetric powers of $T_X$ has global sections, so certainly $T_X$ is not big. For $d=3$, the pseudoeffectivity of $T_X$ is more delicate,
but an intersection computation involving the Segre classes allows to conclude:

\begin{theorem}
	\label{thm:Tangent-hypersurface-not-pseudo-effective}
	For $n\geq 2$, let $X\subset \mathbb{P}^{n+1}$ be a smooth hypersurface.
	Then $T_X$ is pseudoeffective if and only if $d\leq 2$.
\end{theorem}

This statement does not generalise to Fano manifolds that are complete intersections:
del Pezzo surfaces of degree four are anticanonically embedded as a complete intersection
of two quadrics in $\PP^4$, but by Theorem \ref{THM:bigness-pseudo-effectivity-del-Pezzo-surfaces} their tangent bundle is pseudoeffective (see also Remark \ref{remarkschneider}).
Let us also note that Lehmann and Ottem proved that for a Fano hypersurface of degree $d \geq 3$
and $\dim X \leq 5$, the diagonal $\Delta \subset X \times X$ is not big \cite[Thm.1.10]{LehmannOttem2018},
but we do not know if their result can be related to our theorem.

As a first step towards a general theory on higher dimensional Fanos, we consider del Pezzo threefolds with Picard number one. Since these manifolds do not possess any natural fibrations,
we consider the notion of total dual VMRT in Subsection \ref{subsectionVMRT}.
Total dual VMRT provide a natural construction for divisors in $\PP(T_X)$ and we introduce
a technique to compute their cohomology class. As an application we obtain our last result:

\begin{theorem}
	\label{THM:bigness-pseudo-effectivity-del-Pezzo-threefolds}
	Let $X$ be a $3$-dimensional del Pezzo manifold, i.e. a smooth Fano threefold such
	that $-K_X = 2 H$ where $H$ is a Cartier divisor. Then the following holds.
	\begin{enumerate}[(a)]
		\item Assume that $X$ is general in its deformation family if $d=2$. Then $T_X$ is pseudoeffective if and only if $d\geq 4$.
		
		\item $T_X$ is big if and only if $d\geq 5$.
	\end{enumerate}
    Moreover $T_X$ is pseudoeffective if and only if $H^0(X, S^m T_X) \neq 0$ for some $m \in \N$.
\end{theorem}

The most striking feature of this statement is its similarity to the surface
case (Theorem \ref{THM:bigness-pseudo-effectivity-del-Pezzo-surfaces}), although there is no immediate connection between the two situations:
a general element of the linear system $|H|$ is a del Pezzo surface $D$, but
$T_D$ is a {\em subbundle} of $T_X|_D$, so we can not expect any relation between their positivity properties. However, we prove in Lemma \ref{lemmaclassdelpezzo} that the
class of the total dual VMRT depends on the number of $(-1)$-curves in $D$. 

We expect that bigness of the tangent bundle is a rather restrictive property, so it would be very interesting to have a more geometric characterisation.
Greb and Wong \cite{GrebWong2019} show that the affineness of a canonical complex extension
implies that the tangent bundle is big \cite[Cor.4.4]{GrebWong2019}. Our results indicate that many Fano manifolds do not admit such affine  extensions.

{\bf Acknowledgements.} 
The first-named author thanks the Institut Universitaire de France and the A.N.R. project Foliage (ANR-16-CE40-0008) for providing excellent working conditions.
The second-named author is partially supported by China Postdoctoral Science Foundation. The third-named author is grateful to Laurent Manivel for valuable discussions and also thanks his advisor Baohua Fu for many helpful suggestions.

\section{Definitions and basic facts} \label{sectiondefinitions}

We work over $\C$, for general definitions we refer to \cite{Har77}. 
We use the terminology of \cite{KM98} for birational geometry and of \cite{Ko96} for rational curves.
We refer to Lazarsfeld's books for notions of positivity of $\R$-divisors, in particular
\cite[Sect.2.2]{Laz04a} for pseudoeffectivity and the associated cones.
We will use the terminology of $\Q$-twisted vector bundles as explained in
\cite[Sect.6.2]{Laz04b}

Manifolds and varieties will always be supposed to be irreducible. A fibration is a proper surjective map with connected fibres \holom{\varphi}{X}{Y} between normal varieties.

In the whole paper we will use the following notational convention:
given a normal projective variety $X$
and a vector bundle $E \rightarrow X$, we denote by
$\holom{\pi}{\PP(E)}{X}$ the projectivisation in the sense of Grothendieck and by
$\zeta=c_1(\sO_{\PP(E)}(1))$ the tautological class on
$\PP(E)$.

\subsection{Positivity of vector bundles}

\begin{definition}
    Let $X$ be a normal projective variety, and let $E \rightarrow X$ be a vector bundle. We say that $E$ is pseudoeffective if the tautological
    class $\zeta$ is pseudoeffective. The vector bundle is big if
    $\zeta$ is big.
\end{definition}

 \begin{lemma} \cite[Lemma 2.7]{Druel2018}
 	\label{lemmacharacterisepseff}
     Let $X$ be a normal projective variety, and let $E \rightarrow X$ be a vector bundle over $X$. Let $H$ be a big $\Q$-Cartier $\Q$-divisor class.
     Then $E$ is pseudoeffective if and only if for all $c>0$ there exist sufficiently divisible positive integers $i,j \in \N$
     such that $i>cj$
     $$
     H^0(X, S^i E \otimes \sO_X(jH)) \neq 0.
     $$
 \end{lemma}

 \begin{proof}
 The statement is proven in Druel's paper under the assumption that $H$ is an ample Cartier divisor, we follow his argument.

 Since all the properties in the statement are invariant under birational modifications, we can assume without loss of generality that $X$ is smooth.

     If the non-vanishing condition holds, the $\Q$-divisor class $i \zeta + \pi^* jH$ is represented by an effective divisor. Thus the class
     $\zeta+\frac{j}{i} \pi^* H$ is pseudoeffective.
     Since $\frac{j}{i}<\frac{1}{c}$ and $c$ can be arbitrarily large, this shows that $\zeta$ is pseudoeffective.

     Assume now that $\zeta$ is pseudoeffective.

     {\em 1st case. $H$ is ample.} By \cite[Prop.1.45]{KM98} there exists a $m_0 \in \N$ such that $\zeta+ m_0 \pi^* H$ is ample. Since $\zeta$ is pseudoeffective, the class
     $m \zeta+ m_0 \pi^* H$ is big for all $m \in \N$. In particular for a given $c>0$ we can choose $m \in \N$ such that $m > c m_0$ and $m \zeta+ m_0 \pi^* H$ is big. Thus for some $k \in \N$ one has
     $$
     0 \neq H^0(\PP(E), \sO_{\PP(E)}(k (m \zeta+ m_0 \pi^* H)))
     =
     H^0(X, S^{km} E \otimes \sO_X(k m_0)).
     $$
     {\em 2nd case. $H$ is big.}
     By Kodaira's lemma \cite[Cor.2.2.7]{Laz04a} we have
     $H \sim_\Q A + N$ where $A$ is ample and $N$ is effective. Applying the first case to $A$ we see that for $c>0$ we can find sufficiently divisible $i,j$ such that $i>cj$ and
     $H^0(X, S^i E \otimes \sO_X(jA)) \neq 0$.
     Since $j$ is sufficiently divisible and $N$ is effective, we have an injection
     $$
     H^0(X, S^i E \otimes \sO_X(jA))
     \hookrightarrow
     H^0(X, S^i E \otimes \sO_X(jH)).
     $$
     This completes the proof of the lemma.
 \end{proof}

 \begin{lemma}\label{lemmasubstractbig}
  Let $X$ be a normal projective variety, and let $E \rightarrow X$ be a vector bundle over $X$.
 Then $E$ is big if and only if there exists a big $\Q$-Cartier $\Q$-divisor class $H$ on $X$ such that $\zeta-\pi^* H$ is pseudoeffective.
 \end{lemma}

\begin{proof}
    Since bigness is an open property one implication is trivial. Assume now that $\zeta-\pi^* H$
    is pseudoeffective for some big $\Q$-Cartier $\Q$-divisor class $H$.
    By Kodaira's lemma \cite[Cor.2.2.7]{Laz04a} we have $H \sim_\Q A + N$ where $A$ is ample and $N$ is effective. By \cite[Prop.1.45]{KM98} there exists a $m_0 \in \N$ such that $\zeta+ m_0 \pi^* A$ is ample. In particular the sum
    $$
    m_0 (\zeta-\pi^* H) + (\zeta+ m_0 \pi^* A)
    = (m_0+1) \zeta - m_0 \pi^* N
    $$
    is big. Since $ - m_0 \pi^* N$ is antieffective, this implies that  $(m_0+1) \zeta$ is big.
\end{proof}

 \begin{corollary} \label{corollarygodown}
     Let $\holom{\mu}{X}{X'}$ be a birational morphism between projective manifolds.
     If $T_X$ is pseudoeffective, then $T_{X'}$ is pseudoeffective.
 \end{corollary}

\begin{proof}
Denote by $Z \subset X'$ the image of the exceptional locus. For every $i \in \N$ the direct image $\mu_* (S^i T_X)$ is a torsion-free sheaf that coincides with
$S^i T_{X'}$ on $X' \setminus Z$. Since $S^i T_{X'}$ is reflexive and $Z$ has codimension at least two, we obtain an injective morphism
$$
\mu_* (S^i T_X) \hookrightarrow S^i T_{X'}.
$$
Fix $H$ an ample Cartier divisor on $X'$. Let $c>0$ be an arbitrary positive number.
Since $T_X$ is pseudoeffective, we know by Lemma \ref{lemmacharacterisepseff} that there exist
sufficiently divisible positive integers $i,j \in \N$
     such that $i>cj$ and
     $$
     H^0(X, S^i T_X \otimes \sO_X(j \mu^* H)) \neq 0.
     $$
     By the projection formula and the injection constructed above, this implies that
     $$
     H^0(X', S^i T_{X'} \otimes \sO_{X'}(jH)) \neq 0.
     $$
     By Lemma \ref{lemmacharacterisepseff} this shows that $T_{X'}$ is pseudoeffective.
\end{proof}

Let $X$ be a projective manifold of dimension $n$, and let $\alpha$ be an $\R$-divisor class on $X$ that is pseudoeffective.
By  \cite[Thm.3.12]{Bou04} we have a divisorial Zariski decomposition
\begin{equation} \label{Bdecomposition}
\alpha = \sum_{j=1}^r \lambda_j D_j + P,
\end{equation}
where the $D_j$ are prime divisors on $X$, the coefficients $\lambda_j \in \R^+$ and $P$ is a
pseudoeffective class which is modified nef \cite[Defn. 2.2.,Prop.3.8]{Bou04}.
By \cite[Prop.2.4]{Bou04} this implies that for every prime divisor $D \subset X$ the restriction $P|_D$ is pseudoeffective. In particular for any collection $H_1, \ldots, H_{m-2}$ of nef divisor classes, one has
\begin{equation} \label{inequalitymodifiednef}
P^2 \cdot H_1 \cdot \ldots H_{m-2} \geq 0.
\end{equation}

\begin{lemma} \label{lemmaextremalclass}
    Let $X$ be a projective manifold, and let $\alpha$ be a pseudoeffective $\R$-divisor
    class that generates an extremal ray $R$ in the pseudoeffective cone. Then $\alpha$ is modified nef or there exists a unique prime divisor $D_1 \subset X$
    such that $R=\R^+ D_1$.
\end{lemma}

\begin{proof}
    Let $\alpha = D + P$ be the divisorial Zariski decomposition. Since $L$ generates an extremal ray, we have $P=\lambda \alpha$ and $D=(1-\lambda) \alpha$ for some
    $\lambda \in [0,1]$. By uniqueness of the Zariski decomposition we obtain
    $\lambda=1$ or $\lambda=0$. If $\lambda=1$, then $\alpha=P$ is modified nef.
    If $\lambda=0$, we have
    $$
    \alpha= D = \sum_i \lambda_i D_i
    $$
    where $D_i \subset X$ are the prime divisors appearing in the negative part. Since $\alpha$ generates an extremal ray we have $D_i = \lambda_i \alpha$ with
    $\sum_i \lambda_i =1$. By \cite[Prop.3.11(iii)]{Bou04} the classes of the prime divisors $D_i$ are linearly independent, so we obtain $\alpha=D=D_1$. Since $\kappa(D_1)=0$ by \cite[Prop.3.13]{Bou04} the statement follows.
\end{proof}

\begin{corollary} \label{corollaryextremalclass}
Let $X$ be a projective manifold with Picard number one, and let $E \rightarrow X$ be a vector bundle. Let $\alpha$ be the unique $\R$-divisor class on $X$ such that
$\zeta+\pi^* \alpha$ generates an extremal ray in $N^1(\PP(E))$. Then
$\zeta+\pi^* \alpha$ is modified nef or there exists a unique prime divisor
$D_1 \subset \PP(E)$ such that $\R^+ (\zeta+\pi^* \alpha) =\R^+ D_1$.
\end{corollary}

\subsection{Dual variety of minimal rational tangents} \label{subsectionVMRT}

Let $X$ be a uniruled projective manifold of dimension $n$. A \emph{family of rational curves} on $X$ is, by definition, an irreducible component $\mathcal{K}$ of $\Rat(X)$. Each of these families comes equipped with a commutative diagram:
\[
\xymatrix{
\mathcal{U}\ar[r]^{e}\ar[d]_{q}  &  X\\
\mathcal{K}                      &
}
\]
where $\mathcal{U}$ is the normalisation of the universal family over $\mathcal{K}$, the morphism $q$ is a smooth $\PP^1$-fibration and $e$ is the evaluation morphism (see \cite[Sect.II.2]{Ko96} for details). Given a point $x\in e(\mathcal{U})$, we denote by $\mathcal{K}_x$ the normalisation of the set $q(e^{-1}(x))$, and by $\mathcal{U}_x$ the normalisation of its fibre product with $\mathcal{U}$ over $\mathcal{K}$. We say that the family $\mathcal{K}$ is
\begin{itemize}
	\item \emph{dominating} if $e$ is dominant,
	
	\item \emph{locally unsplit} if $\mathcal{K}_x$ is proper for general $x\in X$.
\end{itemize}

{\em
For the rest of the section we consider a family of rational curves $\mathcal K$ that is dominating
and locally unsplit.}

Recall that if $x\in X$ is a general point, then $\mathcal{K}_x$ is smooth. There exists a rational map $\tau_x$ from $\mathcal{K}_x$ to $\PP(\Omega_{X,x})$, called the \emph{tangent map} of $\mathcal{K}$ at $x$, sending the general element of $\mathcal{K}_x$ to its tangent direction at $x$. It is known that $\tau_x$ is a finite and birational morphism \cite{HM04} onto a variety $\mathcal{C}_x\subset \PP(\Omega_{X,x})$ usually called the \emph{variety of minimal rational tangents} (VMRT, for short) of $\mathcal{K}$ at $x$.
The set
$$
\mathcal C := \overline{\bigcup_{x \in X \mbox{\tiny general}} \mathcal{C}_x} \subset \PP(\Omega_X)
$$
is called the total VMRT associated to $\mathcal{K}$. Let $[l]\in \mathcal{K}$ be a general member with normalisation $f\colon\PP^1\rightarrow l$. Then $l$ is a standard rational curve; that is, we have
\[
f^*T_X\cong \sO_{\PP^1}(2)\oplus\sO_{\PP^1}(1)^{\oplus p}\oplus\sO_{\PP^1}^{\oplus(n-p-1)}.
\]
A \emph{minimal section} $\widetilde{l}$ of $\PP(T_X)$ over the minimal rational curve $l$ is a rational curve
$$
\tilde f: \PP^1 \rightarrow \PP(T_X\vert_{l})\subset\PP(T_X)
$$
corresponding to a quotient $f^*T_X\twoheadrightarrow \sO_{\PP^1}$. Consider an irreducible component $\widetilde{\mathcal{K}}\subset \Rat(\PP(T_X))$ containing a minimal section $\widetilde{l}$ of $\PP(T_X)$ over $l$ and the corresponding universal family $\tilde U \rightarrow \tilde K$. Then we have a commutative diagram (see \cite[\S\,4]{OCW16} for more details):
\[
\xymatrix{
\widetilde{\mathcal{K}}\ar[d]_{\widetilde{\pi}} & \widetilde{\mathcal{U}}\ar[l]_{\widetilde{q}}\ar[d]\ar[r]^{\widetilde{e}} & \PP(T_X)\ar[d]^\pi\\
\mathcal{K}        &   \mathcal{U}\ar[l]^q\ar[r]^e   &  X
}
\]

\begin{proposition}\cite[Prop. 4]{OCW16}
	\label{Prop:family of minimal sections}
	With the same notation as above, the variety $\widetilde{\mathcal{U}}$ is smooth at $[\widetilde{l}]$, of dimension $2n-3$. Moreover, let  $\widetilde{f}\colon \PP^1\rightarrow \widetilde{l}$ be the normalisation. Then there exists a nonnegative integer $c\leq p$ such that
	\[\widetilde{f}^*T_{\PP(T_X)}\cong \sO_{\PP^1}(-2)\oplus\sO_{\PP^1}(2)\oplus\sO_{\PP^1}(-1)^{\oplus c}\oplus\sO_{\PP^1}(1)^{\oplus c}\oplus\sO_{\PP^1}^{\oplus (2n-2c-3)}.\]
	The number $c$ will be called the defect of $\widetilde{\mathcal{K}}$ at $\widetilde{l}$.
\end{proposition}

\begin{definition} \label{definitiontotaldual}
	With the same notation as above, denote by $\check{\mathcal{C}}$ the closure of $\widetilde{e}(\widetilde{\mathcal{U}})$ in $\PP(T_X)$.
	The variety $\check{\mathcal{C}}$ is called the total dual VMRT of the family $\mathcal K$.
\end{definition}

The next result justifies the terminology in Definition \ref{definitiontotaldual}:

\begin{proposition}
	\cite[Prop. 5, Cor. 5]{OCW16}
	With the same notation as above, let $x\in X$ be a general point. Then $\check{\mathcal{C}}_x$ is the projectively dual variety of $\mathcal{C}_x$ and the dual defect of $\mathcal{C}_x$ equals to the defect of $\widetilde{\mathcal{K}}$ at $\widetilde{l}$.
\end{proposition}

\begin{remark*}
Let us briefly recall the definition of a projectively dual variety. Let $V$ be a vector space of dimension $N+1$, and let $X\subset\PP(V)$ be a projective variety. We denote by $T_{X,x}$ the tangent space at any smooth point $x\in X$. We can also define the \emph{embedded projective tangent space} $\widehat{T}_{X,x}\subset \PP(V)$ as follows:
\[
\widehat{T}_{X,x}=\PP(T_{v,\textrm{Cone}(X)}),
\]
where $\textrm{Cone}(X)\subset V$ is the affine cone over $X$, $v$ is any non-zero point on the line $\C x$, and we consider $T_{v,\textrm{Cone}(X)}$ as a linear space (it does not depend on the choice of $v$). A hyperplane $H\subset \PP(V)$ is a \emph{tangent hyperplane} of $X$ if $\widehat{T}_{X,x}\subset H$ for some smooth point $x\in X$. The closure $\check{X}\subset\PP(V^*)$ of the set of all tangent hyperplanes is called the \emph{projectively dual variety} of $X$. The \emph{dual defect} of $X$ is defined as
\[
\codim_{\PP^(V^*)}\check{X}-1.
\]
\end{remark*}

One of the main tools in this paper is to compute the divisor class of a total dual VMRT $\check{\mathcal{C}}$. This requires to go beyond the definition of $\check{\mathcal{C}}$ as a closure and consider the completion of the family of rational curves:

Let $\overline{\mathcal{K}}$ be the normalisation of the closure of $\mathcal{K}$ in $\Chow{X}$, and denote by $\overline{q}\colon \overline{\mathcal{U}}\rightarrow \overline{\mathcal{K}}$ the normalisation of the universal family and by $\overline{e}\colon \overline{\mathcal{U}}\rightarrow X$
the natural map extending the evaluation morphism.
Since a general member of $\mathcal K$ is a standard rational curve, it is immersed. Thus the set $\mathcal K_0 \subset \overline{\mathcal{K}}$ parametrising immersed rational curves is not empty and the morphism
\[
\overline{e}^*\Omega_X\rightarrow \Omega_{\overline{\mathcal{U}}}\rightarrow \Omega_{\overline{\mathcal{U}}/\overline{\mathcal{K}}}
\]
is generically surjective and even surjective on $\fibre{\overline{q}}{\mathcal K_0}$.
Denote by $\sT_{\overline{\mathcal{U}}/\overline{\mathcal{K}}}$ the dual
$\Omega_{\overline{\mathcal{U}}/\overline{\mathcal{K}}}^* \rightarrow T_{\overline{\mathcal{U}}}$. Then, dualising the morphisms above,
 one gets an exact sequence of sheaves
\begin{equation}
\label{Equa:sequnece of tangents of universal family}
	0\rightarrow \sT_{\overline{\mathcal{U}}/\overline{\mathcal{K}}}\rightarrow \overline{e}^*T_X\rightarrow \sG\rightarrow 0,
\end{equation}
where $\sG$ denotes the cokernel. For clarity's sake let us note that
$\sT_{\overline{\mathcal{U}}/\overline{\mathcal{K}}}$ is in general neither saturated in $\overline{e}^*T_X$, nor in  $T_{\overline{\mathcal{U}}}$.

Let $P$ be the unique component of $\PP(\sG)\subset \PP(\overline{e}^*T_X)$ dominating $\mathcal{\overline{U}}$. Denote by $\overline{\pi}$ the natural projection $\PP(\overline{e}^*T_X)\rightarrow \overline{\mathcal{U}}$, then we have a Cartesian diagram:
	\[
	\xymatrix{
		\PP(\overline{e}^*T_X)\ar[d]_{\overline{\pi}}\ar[r]^{\widetilde{e}}   &  \PP(T_X)\ar[d]^{\pi}\\
		\overline{\mathcal{U}}\ar[r]_{\overline{e}} & X
    }
	\]

We will now focus on a special case:

\begin{proposition} \label{propositionclassVMRT}
	With notation as above, assume furthermore that
	\begin{itemize}
	    \item $p=0$ ;
	    \item $\sG$ is locally free in codimension one.
	\end{itemize}
Then we have $\widetilde{e}_*(P)=\check{\mathcal{C}}$ and
	\[
	[\check{\mathcal{C}}]=\deg(\overline{e}) \zeta - \pi^*\overline{e}_* c_1(\sT_{\overline{\mathcal{U}}/\overline{\mathcal{K}}}).
	\]
\end{proposition}

\begin{remark*}
If $p=0$, the morphism $\overline{e}$ is generically finite, so its degree $\deg(\overline{e})$ is
well-defined. Since $\overline{U}$ is only assumed to be normal, the first Chern class $c_1(\sT_{\overline{\mathcal{U}}/\overline{\mathcal{K}}})$ is only the Weil divisor class
associated to the reflexive sheaf $\sT_{\overline{\mathcal{U}}/\overline{\mathcal{K}}}$.
Since push-forwards and flat pullbacks are well-defined for Weil divisor classes, this does not pose any difficulty.
\end{remark*}

\begin{proof}
	As $p=0$, the dual defect $c$ vanishes by Proposition \ref{Prop:family of minimal sections} and $\check{\mathcal{C}}$ is a divisor.
	Moreover the evaluation map is \'etale in a neighbourhood of a standard rational curve $l \in \mathcal K$, so the cokernel $\sG|_l$ coincides with the quotient defined by
	the injection $T_l \rightarrow f^* T_X$. In particular $\PP(\sG|_l)$ is the projective dual
	corresponding to the point $\PP(\Omega_l) \subset \PP(\Omega_X|_l)$.
	This already shows that $\widetilde{e}(P)\subset\check{\mathcal{C}}$. Since
	$\check{\mathcal{C}}$ is irreducible we have a set-theoretical equality.
	
	Moreover, since the tangent map $\tau_x$ is birational \cite{HM04}, and $\mathcal C_x$ is finite, it is injective. Thus its projective dual
	consists of $\deg(\overline{e})$ hyperplanes and the map $P \rightarrow \check{\mathcal{C}}$
	is birational. Thus we have $\widetilde{e}_*(P)=\check{\mathcal{C}}$.

Since $\sG$ is locally free in codimension one, the variety $P$ is the only divisorial component of $\PP(\sG)\subset\PP(\overline{e}^*T_X)$. Moreover, by restricting the sequence \eqref{Equa:sequnece of tangents of universal family} to the locally free locus of $\sG$, we obtain
	\[
	[P] = c_1(\sO_{\PP(\overline{e}^*T_X)}(1)) - \overline{\pi}^*c_1(\sT_{\overline{\mathcal{U}}/\overline{\mathcal{K}}}).
	\]
	Since $\widetilde{e}^* \zeta = c_1(\sO_{\PP(\overline{e}^*T_X)}(1))$, this yields
	\begin{align*}
		[\check{\mathcal{C}}]=[\widetilde{e}_*P] & = \widetilde{e}_*c_1(\sO_{\PP(\overline{e}^*T_X)}(1)) - \widetilde{e}_*\overline{\pi}^*c_1(\sT_{\overline{\mathcal{U}}/\overline{\mathcal{K}}})\\
		& =\deg(\widetilde{e}) \zeta - \pi^*\overline{e}_* c_1(\sT_{\overline{\mathcal{U}}/\overline{\mathcal{K}}}).
	\end{align*}
	Then we conclude by the fact that $\deg(\widetilde{e})=\deg(\overline{e})$.
\end{proof}

As an immediate application, one can easily derive the following criterion for bigness of tangent bundle of Fano manifolds with Picard number one.

\begin{corollary} \label{corollarytangentnotbig}
	In the situation of Proposition \ref{propositionclassVMRT},
	assume that $\pic(X) \cong \Z H$ with $H$ ample. If
	$- \pi^*\overline{e}_* c_1(\sT_{\overline{\mathcal{U}}/\overline{\mathcal{K}}})=mH$
	for some $m \geq 0$, then $T_X$ is not big.
\end{corollary}

\begin{proof}
	If $T_X$ is big, the class $\zeta$ lies in the interior of the pseudoeffective cone of $\PP(T_X)$.
	Since this cone is generated by the classes of prime divisors
	there exists a $\varepsilon>0$ and $\lambda \in \N$ such that
	the linear system $| \lambda(\zeta-\varepsilon \pi^* H) |$ contains a prime divisor $D$. Let $\tilde l \subset \PP(T_X)$ be a minimal section,	then $D \cdot \tilde l<0$. Thus the minimal sections are contained in $D$. Since $p=0$, the total dual VMRT $\check{\mathcal{C}}$ is dominated by minimal sections. Thus one obtains $D=\check{\mathcal{C}}$. Yet, by Proposition \ref{propositionclassVMRT} this contradicts $m \geq 0$.
\end{proof}

While the first condition in Proposition \ref{propositionclassVMRT} is rather straightforward, the second condition is more technical. It can be verified easily in two cases:

\begin{remark} \label{remarkVMRTunsplit}
In the situation of Proposition \ref{propositionclassVMRT}, suppose that
 $\mathcal{K}$ is an unsplit dominating family of minimal rational curves. Then we have $\overline{\mathcal{K}}=\mathcal{K}$ and $\Omega_{\overline{\mathcal{U}}/\overline{\mathcal{K}}}=\Omega_{\mathcal{U}/\mathcal{K}}$ is locally free. In particular, $\sT_{\overline{\mathcal{U}}/\overline{\mathcal{K}}}$ is locally free and the locus where $\sG$ is not locally free corresponds to the points
 where the tangent map of
 $$
 e|_{U_{[l]}} : T_{l} \rightarrow (e^* T_X)|_{l}
 $$
 is not injective. By generic smoothness this locus has codimension at least one in $l$.
 Since a general minimal rational curve is immersed, it is empty for $[l] \in \mathcal{K}$ general.
 Thus the locus where $\sG$ is not locally free has codimension at least two in $\sU$.
\end{remark}

\begin{corollary} \label{corollaryconicbundle}
	Let $X$ be a projective manifold admitting a conic bundle structure $f\colon X\rightarrow Y$ over a projective manifold $Y$. Denote by $\check{\mathcal{C}}$ the total dual VMRT associated to the fibres of $f$. Then we have
	\[
	[\check{\mathcal{C}}]\sim\zeta+ \pi^* K_{X/Y}.
	\]
\end{corollary}

\begin{proof}
The fibres of $f$ define a dominating, locally unsplit family with $p=0$, moreover we are
in the special case where $\overline{\mathcal{U}}=X$ and $\overline{\mathcal{K}}=Y$.
Yet a conic bundle has reduced fibres in codimension one \cite[Prop.1.8.5]{Sar82}, so
$\Omega_{X/Y}^*$ is a subbundle of $T_X$ in codimension one and we have 
$$
-c_1(\sT_{\overline{\mathcal{U}}/\overline{\mathcal{K}}})=
-c_1(T_{X/Y})
=
K_{X/Y}.
$$
Moreover, the quotient
$\sG$ is locally free in codimension one. Thus, we conclude by applying Proposition \ref{propositionclassVMRT}.
\end{proof}

\section{del Pezzo surfaces}

Let $X$ be a smooth del Pezzo surface of degree $d \geq 3$. Then the anticanonical divisor $-K_X$
is very ample and defines an embedding
$$
X \hookrightarrow \PP(H^0(X, -K_X)).
$$
Since $-K_X$ is the restriction of the hyperplane class, a line (resp. conic) on $X$
is simply a smooth rational curve $C$ such that $-K_X \cdot C=1$ (resp. $-K_X \cdot C=2$).
By adjunction a line on $X$ is a (-1)-curve, while a conic has $C^2=0$.
In particular it defines a basepoint-free pencil, so every conic on $X$ induces a conic bundle
structure $f: X \rightarrow \PP^1$.
In this section we will use classical geometry to describe the total dual VMRT
associated to these conic bundles. Theorem \ref{THM:bigness-pseudo-effectivity-del-Pezzo-surfaces} will be an immediate consequence
of these computations.

The following result is well known for experts, we include a complete proof for the reader's convenience.

\begin{lemma}
	\label{lemma:freeness-twist-tangent-bundle}
	Let $n$ be a positive integer. Then the vector bundle $(\wedge^{r}\Omega_{\mathbb{P}^n})(r+1)$ is globally generated for $0\leq r\leq n$. In particular, if $X\subset \mathbb{P}^{n}$ be a submanifold of dimension $m\geq 1$, then $T_X(m)\otimes \sO_X(K_X)$ is globally generated.
\end{lemma}

\begin{proof}
	For simplicity, let us denote $\wedge^{r}\Omega_{\PP^n}$ by $\Omega^{r}_{\PP^n}$ and denote the vector bundle $\sO_{\mathbb{P}^{n}}(1)^{\oplus (n-1)}$ by $E$. Let $l\subset \mathbb{P}^{n}$ be an arbitrary projective line.  Then we have the following twisted Koszul complex (see for instance \cite[Appendices, B.2]{Laz04a})
	\[0\rightarrow (\wedge^{n-1} E)\otimes \Omega^r_{\mathbb{P}^n}(r+1)\rightarrow \cdots\rightarrow E\otimes\Omega^r_{\mathbb{P}^n}(r+1)\rightarrow \Omega^r_{\mathbb{P}^n}(r+1)\rightarrow \Omega^r_{\mathbb{P}^n}(r+1)\vert_l\rightarrow 0.\]
	By Bott's formula, for any $1\leq j\leq n-1$, we have
	\[H^{j}(\PP^n, \Omega^{r}_{\mathbb{P}^n}(r+j+1))=0.\]
	As a consequence, the associated spectral sequence shows that the restriction
	\[H^0(\mathbb{P}^n,\Omega^r_{\mathbb{P}^n}(r+1))\rightarrow H^0(l,\Omega^r_{\mathbb{P}^n}(r+1)\vert_l)\]
	is surjective, see \cite[Lemma B.1.3]{Laz04a}. Since $\Omega^r_{\mathbb{P}^n}(r+1)\vert_l$ is globally generated and $l$ is arbitrary, it follows that $\Omega^r_{\mathbb{P}^n}(r+1)$ itself is globally generated.
	
	Finally, since $X$ is a submanifold of $\PP^n$, we have a surjection
	$
	\Omega^{m-1}_{\PP^n}\vert_X\rightarrow \Omega_X^{m-1}.
	$
	Now the last assert follows immediately from the duality $\Omega_X^{m-1}\cong T_X\otimes\sO_X(K_X)$.
\end{proof}

As a consequence, we have the following description of the nef cone of $\PP(T_X)$ for $X$ being a smooth cubic hypersurface.

\begin{corollary}
	Let $X\subset \PP^{n+1}$ be a smooth cubic hypersurface of dimension $n\geq 2$. Then $\zeta+\pi^*H$ is nef, where $H$ is a hyperplane section of $X$. Moreover, if $n\geq 3$, then the nef cone of
	$\PP(T_X)$ is generated by $\pi^* H$ and $\zeta+\pi^* H$.
\end{corollary}

\begin{proof}
	Since $X$ is an $n$-dimensional smooth cubic hypersurface, we have $\sO_X(K_X)\cong \sO_X(-n+1)$. By Lemma \ref{lemma:freeness-twist-tangent-bundle} above, the vector bundle $T_X(1)$ is globally generated. In particular, $\zeta+\pi^*H$ is nef. On the other hand, note that there exist lines of second type on $X$, i.e. a projective line $l\subset X$ such that
	\[
	T_X\vert_l\cong \sO_{\PP^1}(2)\oplus \sO_{\PP^1}(1)^{\oplus n-2}\oplus \sO_{\PP^1}(-1).
	\]
	Let $\overline{l}\subset \PP(T_X\vert_l)$ be the section corresponding to the quotient $T_X\vert_l\rightarrow \sO_{\PP^{1}}(-1)$. Then we have
	\[
	(\zeta+\varepsilon \pi^*H)\cdot\overline{l}=-1+\varepsilon.
	\]
	In particular, $\zeta+\varepsilon\pi^*H$ is nef only if $\varepsilon \geq 1$. If $n\geq 3$, then we have $\rho(\PP(T_X))=2$ and the result follows immediately.
\end{proof}

\subsection{Cubic surfaces}

Let $X$ be a smooth del Pezzo surface of degree $3$. Then $X$ is isomorphic to a smooth cubic surface in $\PP^3$ with $-K_X$ a hyperplane section. Moreover, it is well known that $X$ contains exactly $27$ lines, namely $E_i$ $(i=1,\dots,27)$.

Given a line $l:=E_i$ we can find a MMP, i.e. birational morphism $X \rightarrow \PP^2$ such that $X$ is the blow-up of six general points $p_1, \ldots, p_6$
and $l$ is the exceptional divisor over $p_1$. Anticanonical divisors on $X$ correspond to
cubics in $\PP^2$ passing through the points $p_1, \ldots, p_6$.
The union of the strict transform of a conic passing through  $p_1, \ldots, p_5$ and the strict transform of a line passing through $p_1$ and $p_6$ defines an effective divisor linearly equivalent to $-K_X-l$. Since this effective divisor
identifies to two concurrent lines (via the embedding $X \hookrightarrow \PP^3$)
one verifies immediately that $-K_X-l$ is nef and $(-K_X-l)^2=0$. A Riemann-Roch computation now shows that the
the complete linear system $\vert-K_X-l\vert$ is base point free and defines a conic bundle structure
\[
f_l\colon X\rightarrow \PP^1,
\]
whose fibres are the members in $\vert-K_X-l\vert$. Since $-K_X$ is ample and has degree two on the fibres, the fibration $f$ is a conic bundle. Since $X$ is smooth,
the fibration has no multiple fibres (apply \cite[III, Lemma 8.3]{BHPV04}).
Since $\rho(X)=7$, the fibration thus has exactly 5 singular fibres, each consisting
of a pair of concurrent lines.

The divisor classes $(-K_X-E_i)_{i=1, \ldots, 27}$ are linearly independent, so we obtain in this way 27 distinct conic bundle structures  $f_i\colon X\rightarrow \PP^1$.

\begin{proposition} \label{proposition-degree-3}
	Let $X$ be a smooth cubic surface. Then $T_X$ is not pseudoeffective.
\end{proposition}

\begin{remark*}
Slightly weaker statements are shown in  \cite[Thm.5.2]{Mal20} and \cite{BogomolovDeOliveira2008}.
\end{remark*}

\begin{proof}
Let $E_i$ be one of the 27 lines on $X$, and denote by
$f_i\colon X\rightarrow \PP^1$ the corresponding conic bundle structure.
Since $f_i$ has no multiple fibres, we know by
Corollary \ref{corollaryconicbundle} that the class of the associated total dual VRMT
is
\begin{equation} \label{eqnconicVMRT}
  [\check{\mathcal{C}}_i]
  =
  \zeta + \pi^* (K_X+2F_i),
\end{equation}
where $F_i$ is a general $f_i$-fibre.

{\em Step 1. We claim that for every $\lambda_i <\frac{1}{4}$, the restriction of
$\zeta-\lambda_i [\check{\mathcal{C}}_i]$ to $\check{\mathcal{C}}_i$ is not pseudoeffective.}

Recall that $c_1(X)^2=3$ and $c_2(X)=9$, so we have
	$	\zeta^3 = c_1(X)^2-c_2(X)=-6$.
		By Lemma \ref{lemma:freeness-twist-tangent-bundle} we know that the class $\zeta+\pi^* H$ is nef.
	Since
	$$
	\zeta^3=-6, \ \pi^* H \cdot \zeta^2=3, \  \pi^* F \cdot  \zeta^2=2, \  \pi^* (H \cdot F) \cdot \zeta = 2
	$$
	we obtain
	$$
	\zeta \cdot [\check{\mathcal{C}}_i] \cdot (\zeta+\pi^* H) = -1
	$$
	and
	$$
	[\check{\mathcal{C}}_i]^2 \cdot (\zeta+\pi^* H) = -4.
	$$
	In particular for $\lambda_i < \frac{1}{4}$, we have
	$$
	(\zeta-\lambda_i[\check{\mathcal{C}}_i]) \cdot [\check{\mathcal{C}}_i] \cdot (\zeta+\pi^* H) = -1+4\lambda_i < 0.
	$$
	Since $\zeta+\pi^* H$ is nef, this implies that the restriction of $\zeta-\lambda_i[\check{\mathcal{C}}_i]$
	to $\check{\mathcal{C}}_i$ is not pseudoeffective for $\lambda_i < \frac{1}{4}$.

{\em Step 2. Conclusion.}	
	Arguing by contradiction we assume that $\zeta$ is pseudoeffective. By what precedes
	we know that $\zeta|_{\check{\mathcal{C}}_i}$ is not pseudoeffective, so $[\check{\mathcal{C}}_i]$ appears in the negative part of the divisorial Zariski decomposition
	\begin{equation} \label{eqnZdecomposition}
	\zeta = \sum_{j=1}^m c_j [D_j] + P
	\end{equation}
	where $c_j>0$ and $P$ is the positive part, in particular $P|_D$ is pseudoeffective for every prime divisor $D \subset \PP(T_X)$. Up to renumbering we can assume without loss of generality that $\check{\mathcal{C}}_i=D_i$. Since
	$$
	(\zeta-c_i[\check{\mathcal{C}}_i])|_{\check{\mathcal{C}}_i}
	= \sum_{j=2}^m c_j [D_j \cap \check{\mathcal{C}}_i] + P|_{\check{\mathcal{C}}_i}
	$$
	is pseudoeffective, we know by Step 1 that $c_i \geq \frac{1}{4}$
	for each of the divisors $\check{\mathcal{C}}_1, \ldots, \check{\mathcal{C}}_{27}$.
	In particular the class
	$\zeta - \sum_{i=1}^{27} \frac{1}{4} [\check{\mathcal{C}}_i]
	$
	is pseudoeffective. Yet if $l \subset \PP(T_X)$ is a fibre of the projection $\pi$, then
	$$
	(\zeta - \sum_{i=1}^{27} \frac{1}{4} [\check{\mathcal{C}}_i]) \cdot l
	= 1 - 27 \cdot \frac{1}{4}<0.
	$$
	Since $l$ is a mobile curve we have reached a contradiction.
\end{proof}

\subsection{Del Pezzos of degree four}

Let $X$ be a smooth del Pezzo surface of degree $d=4$. Then $X$ is isomorphic to a complete intersection of two quadrics in $\PP^4$ (this is classical, see \cite[Sect.V.1]{Dem76} for a detailed exposition).

\begin{proposition}\cite[\S\,3]{KST89}
	\label{Prop:lines-conics-degree-4}
	Let $X$ be a smooth del Pezzo surface of degree $4$.
	\begin{enumerate}[(a)]
		\item There exist five pairs of pencils of conics $\vert C_i\vert$ and $\vert C_i'\vert$ such that
		\[
		C_i+C_i'\in\vert-K_X\vert.
		\]
		
		\item For each $i \in \{1, \ldots, 5\}$, both $\vert C_i\vert$ and $\vert C_i'\vert$ contain exactly four degenerate members consisting of two lines contained $X$. The entire set of those lines contains exactly all the $16$ lines on $X$.
	\end{enumerate}
\end{proposition}

\begin{proposition} \label{proposition-degree-4}
	Let $\check{\mathcal{C}}_i$ and $\check{\mathcal{C}}_i'$ be the dual VMRTs associated to $\vert C_i\vert$ and $\vert C'_i\vert$, respectively. Then we have
	\[
	[\check{\mathcal{C}}_i]+[\check{\mathcal{C}}_i']=2\zeta.
	\]
	In particular $T_X$ is pseudoeffective, but it is not big.
\end{proposition}

This statement was shown independently by Mallory \cite[Thm.5.6]{Mal20} with a completely different technique.

\begin{proof}
By Proposition \ref{Prop:lines-conics-degree-4} the conic bundle structures have no multiple fibres,
so by Corollary \ref{corollaryconicbundle} the class of $\check{\mathcal{C}}_i$ (resp. $\check{\mathcal{C}}_i'$)
is $\zeta+\pi^* (K_X+2 C_i)$ (resp. $\zeta+\pi^* (K_X+2 C_i')$). By Proposition \ref{Prop:lines-conics-degree-4} this implies that
$$
[\check{\mathcal{C}}_i]+[\check{\mathcal{C}}_i']=2\zeta.
$$
Thus we have $\kappa(\zeta) \geq 0$, in particular $T_X$ is pseudoeffective.

Arguing by contradiction we assume that $T_X$ is big. Then there exist positive integers $a$, $b$ and an effective divisor $D\subset \PP(T_X)$ such that
	\[
	D\sim a\zeta+b\pi^*K_X.
	\]
	We can assume without loss of generality that for a given $b \in \N$, we have chosen $a \in \N$ minimal with this property.
	
	For each $i \in \{1, \ldots, 5\}$, the total dual VMRTs $[\check{\mathcal{C}}_i]$ and $[\check{\mathcal{C}}_i']$ are dominated by rational curves $B$ such that $\zeta \cdot B=0$
	(see the construction of minimal sections in Subsection \ref{subsectionVMRT}).
	Since $(a\zeta+b\pi^*K_X) \cdot B=-2b<0$, it
	follows that $\check{\mathcal{C}}_i$ and $\check{\mathcal{C}}_i$ are contained in the support of $D$. In particular, we have
	\[
	(a-10)\zeta+b\pi^*K_X\sim D'=D-\sum_{i=1}^5 (\check{\mathcal{C}}_i+\check{\mathcal{C}}_i')\geq 0.
	\]
	Yet this contradicts the minimality of $a$.
\end{proof}

\begin{remark*}
Since the anticanonical class of $\PP(T_X)$ is $2 \zeta$, Proposition \ref{proposition-degree-4} shows that for a del Pezzo surface of degree four, the variety $\PP(T_X)$ has an effective (reducible) anticanonical divisor. In fact much more is true: the anticanonical system
contains smooth elements $S \subset \PP(T_X)$, these are K3 surfaces of degree $8$. The base locus of $-K_{\PP(T_X)}$ consists exactly of the 16 curves $l_i'$ defined by quotients $T_X|_{l_i} \rightarrow \sO_{\PP^1}(-1)$ where $l_i \subset X$ is a line in $X$. These 16 curves are disjoint,
so every smooth member of the anticanonical system is a K3 surface of Kummer type (see \cite[Remark 8.6.9]{Dol12} for some classical constructions).
\end{remark*}

\subsection{Del Pezzos of degree five}

Let $X$ be a smooth del Pezzo surface of degree $5$, and let
$\mu\colon Y\rightarrow X$ be the blow-up at a general point $x \in X$.
Then $Y$ is a del Pezzo surface of degree four, and the exceptional divisor $E$ is a line
(via the embedding $Y \hookrightarrow \PP^4$). By Proposition \ref{Prop:lines-conics-degree-4},
for every pair of pencils of conics on $Y$, namely $\vert C_i\vert$ and $\vert C'_i\vert$, there exists exactly one which contains $E$ as a component of one of its degenerate members.
We may assume that it is $\vert C_i\vert$, then the push-forward is a pencil of conics $\vert C_i\vert$ on $X$. Thus we obtain five distinct pencils of conics on $X$. Since $X$ is the blow-up of $\PP^2$ in four general points, it is classically known that its conics arise as strict transforms of lines passing through one point or smooth conics passing through the four points. In particular if
$C_i$ and $C_j$ are conics belonging to distinct pencils, we have $C_i \cdot C_j=1$. With this information in mind an intersection computation shows that
\begin{equation} \label{equation-5-conics}
   - 2 K_X = \sum_{i=1}^5 C_i.
\end{equation}

\begin{proposition} \label{proposition-degree-5}
	Let $X$ be a smooth del Pezzo surface of degree $5$. Then there exists an effective divisor $D\subset \PP(T_X)$ such that $D\sim 5\zeta+\pi^*K_X$. In particular $T_X$ is big.
\end{proposition}

\begin{proof}
Denote by $\holom{f_i}{X}{\PP^1}$ the conic bundle structures defined by one of the five pencils of conics
$\vert C_i\vert$.
By Proposition \ref{Prop:lines-conics-degree-4} the conic bundle structures have no multiple fibres,
so by Corollary \ref{corollaryconicbundle} the class of the total dual VMRT $\check{\mathcal{C}}_i$
is $\zeta+\pi^* (K_X+2 C_i)$. By Equation \eqref{equation-5-conics} this implies that
$$
\zeta - \pi^* (\frac{-K_X}{5}) = \frac{1}{5} \sum_{i=1}^5 [\check{\mathcal{C}}_i]
$$
is pseudoeffective.
Since $-K_X$ is ample, Lemma \ref{lemmasubstractbig} shows that $\zeta$ is big.
\end{proof}

\begin{proof}[Proof of Theorem \ref{THM:bigness-pseudo-effectivity-del-Pezzo-surfaces}]
    By Proposition \ref{proposition-degree-3} and Corollary \ref{corollarygodown} we know
    that $T_X$ is not pseudoeffective if $d \leq 3$. If $d \geq 6$, then $X$ is toric, so $T_X$
    is big by \cite[Cor.1.3]{Hsi15}. The remaining cases $d=4$ and $d=5$ are settled by the Propositions \ref{proposition-degree-4} and \ref{proposition-degree-5}.
\end{proof}

\section{Twisted symmetric vector fields on hypersurfaces}

\subsection{Plethysm and vanishing theorem}

Let us recall the following results related to the existence of symmetric differential forms.

\begin{theorem}\cite{Schneider1992}
	\label{thm:Schneider's-vanishing}
	Let $X\subset \PP^N$ be a projective manifold of dimension $n$ with $n>\frac{N}{2}$. Then for any integers $m$, $k$ such that $k\geq m+1$, we have
	\[
	H^0(X,{\rm Sym}^k\Omega_X\otimes \sO_X(m))=0.
	\]
\end{theorem}

\begin{remark} \label{remarkschneider}
    Let $X \subset \PP^4$ be an anticanonically embedded del Pezzo surface of degree four. By Proposition \ref{proposition-degree-4} we have
    $$
      H^0(X,{\rm Sym}^2 T_X) \cong  H^0(X,{\rm Sym}^2\Omega_X \otimes \sO_X(2)) \neq 0.
    $$
    Indeed we are exactly in the limit case where Schneider's theorem does not apply
    (we have $2=n=\frac{N}{2}$ and $k=m=2$). This explains the difference between hypersurfaces and complete intersections of higher codimension.
\end{remark}

\begin{theorem}\cite[Theorem B]{BogomolovDeOliveira2008}
	\label{thm:Bogomolov-Oliveira-vanishing}
	Let $X\subset \PP^{n+1}$ be a smooth hypersurface of degree $d\geq 3$. If $n\geq 2$, for any positive integer $k$, we have
	\[
	H^0(X,{\rm Sym}^k\Omega_X\otimes \sO_{X}(k))=0.
	\]
\end{theorem}

In particular, if $X\subset \PP^3$ is a smooth surface of degree $d$, using the natural duality $T_X\cong \Omega_X \otimes K_X^* \cong  \Omega_{X}(4-d)$, we have the following vanishing theorem:

\begin{corollary}
	Let $X\subset \PP^3$ be a smooth surface of degree $d$. If $d\geq 3$, then we have
	\[
	H^0(X,{\rm Sym}^k T_X\otimes \sO_X(m-4k+dk))=0.
	\]
	for any integers $m$, $k$ with $k\geq m$ and $k>0$, or equivalently, for any $k>0$, we have
	\[
	H^0(X,{\rm Sym}^k(T_X\otimes \sO_X(d-3)))=0.
	\]
\end{corollary}

\begin{proof}
	The first statement is clear from Theorem \ref{thm:Schneider's-vanishing} and Theorem \ref{thm:Bogomolov-Oliveira-vanishing}. For the second statement, it is enough to note that for any integer $m<k$, the sheaf $\Sym^k T_X\otimes \sO_X(m-4k+dk)$ is a subsheaf of $\Sym^k T_X\otimes \sO_X(dk-3k)$.
\end{proof}

We aim, in this section, to show that the vanishing theorem above still holds for smooth hypersurfaces of higher dimension.

Let us first recall some basic facts about Schur functors, and we refer the reader to \cite[Lecture 4 -- 6]{FH91} for more details. A \emph{partition} of a positive integer $n$ is a sequence $\mu=(\mu_1,\mu_2,...,\mu_m)$ of non-negative integers in decreasing order:
\[
\mu_1\geq\mu_2\geq...\geq\mu_m\geq 0
\]
such that $\sum\limits_{i=1}^{m}\mu_i=n$. For each partition $\mu$ of $n$, we can associate $\mu$ with a Young diagram with $\mu_i$ boxes in the $i$-th row and we call such a Young diagram is of shape $\mu$. For example, the partition $\mu=(4,3,1)$ corresponds to the following Young diagram:
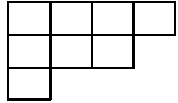
\begin{table}[h]
	\begin{tabular}{|p{2mm}|p{2mm}p{2mm}p{2mm}}
		\hline
		& \multicolumn{1}{l|}{} & \multicolumn{1}{l|}{} & \multicolumn{1}{l|}{} \\ \hline
		& \multicolumn{1}{l|}{} & \multicolumn{1}{l|}{} &                       \\ \cline{1-3}
		&                       &                       &                       \\ \cline{1-1}
	\end{tabular}
\end{table}

Inscribing the integers $1,2,..,n$ into the empty cells (in any order) of the Young diagram of shape $\mu$, we obtain a Young tableau $T$. Let $S_n$ be the symmetric group of degree $n$. We introduc the following two subgroups of $S_n$:
\begin{center}
	$P=\{g\in S_n\,\vert\,g$ preserves each row of $T\}$
\end{center}
and
\begin{center}
	$Q=\{g\in S_n\,\vert\,g$ preserves each column of $T\}$.
\end{center}
Denote by $\C[S_n]$ the group algebra of $S_n$ and let $e_g$ be the element in $\C[S_n]$ corresponding to $g$. Then we can define the following two elements in $\C[S_n]$:
\[
	a_\mu=\sum_{g\in P} e_g \qquad  {\rm and}\qquad   b_{\mu}=\sum_{g\in Q} {\rm sgn(g)} e_g.
\]

Let $V$ be a finite dimensional complex vector space. Note that the symmetric group $S_n$ acts on $V^{\otimes n}$, say on the right, by permuting the factors, so are $a_\mu$ and $b_\mu$. Set $c_\mu=a_\mu b_\mu\in \mathbb{C}[S_n]$, which is called a \emph{Young symmetrizer}.\\

\begin{definition}
	Let $\mu$ be a partition of $n$ and let $V$ be a finite dimensional complex vector space. Denote the image of $c_\mu$ on $V^{\otimes n}$ by $\mathbb{S}_\mu V$. The functor
	\[
	\mathbb{S}_\mu:V\mapsto \mathbb{S}_\mu V
	\]
	is called the Schur functor corresponding to $\mu$. A \emph{plethysm} is a composition of two Schur functors.
\end{definition}

\begin{remark}
	The definition of the Schur functor only depends on the partition $\mu$. By functoriality the definition of Schur functors carries over to vector bundles on projective varieties. For the computation of $\mathbb{S}_{\mu} V$, two cases are easy: for $\mu=(n)$, we have $\mathbb{S}_{\mu}V=\Sym^n V$; for $\mu=(1,...,1)$, we have $\mathbb{S}_{\mu}V=\wedge^n V$.
\end{remark}

\begin{lemma}
	\label{lemma:plethsym}
	Let $k$ be a positive integer and let $V$ be a complex vector space of dimension $n\geq 2$. For any integer $1\leq m\leq n$, the vector space $\mathbb{S}_{\mu(m)} V$ is a direct summand of $\Sym^k(\wedge^m V)$, where $\mu(m)$ is the partition $(k,\dots,k)$ with weight $mk$. In particular, for any positive integer $k$, we have
	\[
	{\rm Sym}^k(\wedge^{n-1}V)=\mathbb{S}_{\mu(n-1)}V.
	\]
\end{lemma}

\begin{proof}
	By the reduction lemma \cite[Lemma A.3]{ManivelMichalek2015}, the vector space $\mathbb{S}_{\mu(m)}V$ appears as a direct summand of $\Sym^k(\wedge^m V)$ if and only if $\mathbb{S}_{\mu(m-1)} V$ appears as a direct summand of $\Sym^k(\wedge^{m-1}V)$. On the other hand, since $\mathbb{S}_{\mu(1)} V$ is isomorphic to $\Sym^k V$, it follows that $\mathbb{S}_{\mu(m)}V$ is a direct summand of $\Sym^k(\wedge^m V)$.
	
	Now we consider the case where $m=n-1$. As $\dim(\wedge^{n-1}V)=n$, we obtain
	\[
	\dim({\rm Sym}^k(\wedge^{n-1}V))=\binom{n+k-1}{n-1}.
	\]
	On the other hand, by \cite[Theorem 6.3]{FH91}, we have
	\[
	\dim(\mathbb{S}_{\mu(n-1)} V)=\prod_{i=1}^{n-1}\frac{k+n-i}{n-i}=\binom{n+k-1}{n-1}.
	\]
	This finishes the proof.
\end{proof}

Theorem \ref{THM:Vanishing-twisted-symmetric-power-hypersurfaces} is a direct consequence of Lemma \ref{lemma:plethsym} and a vanishing theorem for $T$-symmetric tensor forms due to Br\"uckmann and Rackwitz \cite{BrueckmannRackwitz1990}.

\begin{proof}[Proof of Theorem \ref{THM:Vanishing-twisted-symmetric-power-hypersurfaces}]
	We first note that there exists a natural isomorphism
	\[
	T_X\cong \wedge^{n-1}\Omega_X \otimes K_X^* \cong \wedge^{n-1}\Omega_X \otimes \sO_X(n+2-d).
	\]
	Thus, by Lemma \ref{THM:Vanishing-twisted-symmetric-power-hypersurfaces} above, for any positive integer $k$, we have
	\begin{align*}
		{\rm Sym}^k(T_X\otimes \sO_X(d-3))& \cong {\rm \Sym}^k(\wedge^{n-1}\Omega_X\otimes\sO_{X}(n-1))\\
		&\cong \mathbb{S}_{\mu(n-1)}\Omega_X\otimes \sO_{X}((n-1)k),
	\end{align*}
	where $\mu(n-1)$ is the partition $(k,\dots,k)$ with weight $(n-1)k$. As $d\geq 3$, thanks to \cite[Theorem 4 (iii)]{BrueckmannRackwitz1990}, we have $H^0(X,\mathbb{S}_{\mu(n-1)}\Omega_X\otimes\sO_X(p))=0$ for $p\leq (n-1)k$, which is exactly the desired vanishing theorem.
\end{proof}

\subsection{Smooth cubic hypersurfaces}

While Theorem \ref{THM:Vanishing-twisted-symmetric-power-hypersurfaces} is sufficient
to show that the tangent bundle of a hypersurface of degree $d$ is not pseudoeffective if $d \geq 4$, the case of cubics needs some additional arguments. The following result gives the Segre classes of hypersurfaces in projective spaces.

\begin{lemma} \label{lemmasegreclasses}
		For $n\geq 1$, let $X\subset \PP^{n+1}$ be a smooth hypersurface of degree $d$. Then for $1\leq l \leq n$, the Segre class $s_l(T_X)$ is given by
		\[
		(-1)^l \left(
		\binom{n+l+1}{l}-d\times \binom{n+l}{l-1}
		\right) H^l
		=
		(-1)^l\binom{n+l}{l-1}\left(\frac{n+1}{l}-d+1\right)H^l,
		\]
		where $H$ is the hyperplane class on $X$.
	\end{lemma}
	
	\begin{proof}
		By \cite[Proposition 10.3]{EH16}, we have
		\[
		s(\Omega_X)=\frac{1}{c(\Omega_X)}=\frac{1-dH}{(1-H)^{n+2}}=(1-dH)(\sum_{j=0}^\infty H^j )^{n+2}.
		\]
		Thus the Segre class $s_l(\Omega_X)$ is equal to the term of degree $l$ in the right-hand side. Then a straightforward computation shows that the coefficient of $H^l$ is equal to
		\[
		\binom{n+l+1}{l}-d\times \binom{n+l}{l-1}=\binom{n+l}{l-1}\left(\frac{n+1}{l}-d+1\right).
		\]
		Finally the result follows from the fact that $s_l(T_X)=(-1)^l s_l(\Omega_X)$.
	\end{proof}
	
	\begin{proposition}
		\label{prop:cubic-not-modified-nef}
		Let $X\subset \PP^{n+1}$ be a smooth cubic hypersurface with $n\geq 3$. Then we have
		\begin{equation}
		\label{equation:cubic-intersection-number}
		\zeta^2\cdot(\zeta+\pi^*H)^{2n-3}=\frac{-3^2 \cdot 2^n}{8(2n-1)(n+1)}\binom{2n}{n}.
		\end{equation}
		In particular, by \eqref{inequalitymodifiednef}, the class $\zeta$ is not modified nef.
	\end{proposition}
	
	\begin{proof}
		Note that we have
		$$
			\zeta^2\cdot(\zeta+\pi^*H)^{2n-3} = \sum_{i=0}^n \binom{2n-3}{i}(\pi^*H)^i\cdot \zeta^{2n-i-1}.
		$$
		Since $(\pi^*H)^i\cdot \zeta^{2n-i-1}=(-1)^{n-i} s_{n-i}(T_X)$\footnote{Recall that we projectivise in the sense of Grothendieck.}, we obtain by Lemma \ref{lemmasegreclasses}
		that
		\begin{align*}
				\frac{1}{H^n}	\cdot  \zeta^2\cdot(\zeta+\pi^*H)^{2n-3}
			& = \binom{2n-3}{n} + \sum_{i=0}^{n-1}\binom{2n-3}{i}\left(\binom{2n-i+1}{n-i}-3\binom{2n-i}{n-i-1}\right)\\
			& = \sum_{i=0}^n \binom{2n-3}{i}\binom{2n-i+1}{n-i}-3\sum_{i=0}^{n-1}\binom{2n-3}{i}\binom{2n-i}{n-i-1}
		\end{align*}
		Thus the result follows immediately from the following claim:
		
		{\bf Claim.} Let $n$ be a positive integer $\geq 3$. Then we have
		\begin{equation}
		\label{equation:cubic-intersection-number-positive-part}
		\sum_{i=0}^n \binom{2n-3}{i}\binom{2n-i+1}{n-i} = \frac{3(27 n^2 + 9n - 14) 2^n}{64 (2n-1)(n+1)} \binom{2n}{n}
		\end{equation}
		and
		\begin{equation}
		\label{equation:cubic-intersection-number-negative-part}
		\sum_{i=0}^{n-1} \binom{2n-3}{i}\binom{2n-i}{n-i-1} = \frac{3(3n+2)(3n-1)2^n}{64(2n-1)(n+1)}\binom{2n}{n}.
		\end{equation}

	\textit{Proof of the claim.} We give a complete proof of \eqref{equation:cubic-intersection-number-positive-part} and the same argument can be easily modified to prove \eqref{equation:cubic-intersection-number-negative-part}, so we leave the details to interested reader.
		
		For $n=3$ and $4$, a straightforward computation shows that the values on both sides of \eqref{equation:cubic-intersection-number-positive-part} are $96$ and $681$, respectively. Thus we may assume that $n\geq 5$.
		
		For $0\leq i \leq n$, we claim that the following identity holds:
		\begin{equation*}
		\binom{2n-3}{i}\binom{2n-i+1}{n-i}=\frac{n!(2n-i+1)(2n-i)(2n-i-1)(2n-i-2)}{i! (n-i)! 2n(2n-1)(2n-2)(n+1)}\binom{2n}{n}.
		\end{equation*}
		In fact, for $i=0$ and $n$, it follows from the definition. For $1\leq i\leq n-1$, note that we have
		\begin{align*}
		\binom{2n-3}{i}\binom{2n-i+1}{n-i} & = \frac{(2n-i-2)\cdots(2n-3)}{i!}\cdot \frac{(n+2)\cdots(2n-i+1)}{(n-i)!}\\
		& = (n+2)\cdots(2n-3)\cdot \frac{(2n-i+1)\cdots (2n-i-2)}{i! (n-i)!}\\
		& = \binom{2n}{n}\frac{n!(2n-i+1)(2n-i)(2n-i-1)(2n-i-2)}{(n+1)2n(2n-1)(2n-2) i! (n-i)!}.
		\end{align*}
		Therefore, to prove \eqref{equation:cubic-intersection-number-positive-part}, it suffices to prove that the following identity holds:
		\begin{equation}
		\label{equation:simplified-positive-part}
		\sum_{i=0}^n \frac{n!(2n-i+1)(2n-i)(2n-i-1)(2n-i-2)}{n(n-1)}=\frac{3(27n^2+9n-14)2^n}{16}.
		\end{equation}
		On the other hand, since
		\begin{multline*}
		(2n-i-2)(2n-i-1)2n(2n-i+1) = i^4 - (8n-2)i^3 + (24 n^2 - 12n + 1)i^2 \\
		- (32n^3 - 24n^2 - 4n + 2)i +  (16n^4- 16n^3 - 4n^2 +4n),
		\end{multline*}
		we can apply the formula given in Lemma \ref{lemma:combinatorial-identities} below to check \eqref{equation:simplified-positive-part} directly.
	\end{proof}

\begin{lemma}
	\label{lemma:combinatorial-identities}
	Let $n\in \Z_{>0}$ be a positive integer and let $k\in \Z_{\geq 0}$ be a non-negative integer. We define
	\[
	A(k,n) = \sum_{i=0}^n \frac{i^k}{i! (n-i)!}.
	\]
	Then
	\[
	A(k,n)=
	\begin{dcases}
	\frac{2^n}{n!} & k=0;\\
	\frac{n}{2}\cdot \frac{2^n}{n!} & k=1;\\
	\frac{n(n+1)}{4}\cdot \frac{2^n}{n!} & k=2;\\
	\frac{n^2(n+3)}{8}\cdot \frac{2^n}{n!} & k=3;\\
	\frac{n(n+1)(n^2+5n-2)}{16}\cdot \frac{2^n}{n!} & k=4.
	\end{dcases}
	\]
\end{lemma}

\begin{proof}
	For $n=1$, one can easily check the result by a straightforward computation. Moreover, for $k=0$, it follows from the following simple observation
	\begin{align*}
	2^n=\sum_{i=0}^n \binom{n}{i}=n!\sum_{i=0}^n \frac{1}{i! (n-i)!}.
	\end{align*}
	On the other hand, by the definition of $A(k,n)$, for any pair $(k,n)$ with $k$ positive, we note that the following relation holds:
	\[
	A(k,n)=\left(n\sum_{i=0}\frac{i^{k-1}}{i! (n-i)!}\right)-\left(\sum_{i=0}^{n-1}\frac{(n-i) i^{k-1}}{i! (n-i)!}\right) = n A(k-1,n)-A(k-1,n-1).
	\]
	Then we can conclude by an easy inductive argument.
\end{proof}

We can finally conclude the case of hypersurfaces:

\begin{proof}[Proof of Theorem \ref{thm:Tangent-hypersurface-not-pseudo-effective}]
    For $d\leq 2$, it is well known that $X$ is a homogeneous space. In particular, $T_X$ is actually globally generated. Thus it is sufficient to show that $T_X$ is not pseudoeffective if $d\geq 3$.

	For $d\geq 4$, it follows directly from Theorem \ref{THM:Vanishing-twisted-symmetric-power-hypersurfaces} and Lemma \ref{lemmacharacterisepseff}.
	For cubic surfaces the statement is Proposition \ref{proposition-degree-3} in the previous section.
	For $d=3$ and $n\geq 3$, note first that by Theorem \ref{THM:Vanishing-twisted-symmetric-power-hypersurfaces} and Lemma \ref{lemmasubstractbig}, the tangent bundle is not big. In particular, if $T_X$ is pseudoeffective, by Corollary \ref{corollaryextremalclass}, either $\zeta$ is modified nef or there exists a positive integer $m$ such that $m\zeta\sim D$ for an effective integral Weil divisor $D$. The former case is excluded by Proposition \ref{prop:cubic-not-modified-nef}. In the latter case, we must have $H^0(X,\Sym^m T_X)\not=0$, which contradicts Theorem \ref{THM:Vanishing-twisted-symmetric-power-hypersurfaces} again.
\end{proof}

\section{Del Pezzo threefolds}

\subsection{Total dual VMRT of del Pezzo threefolds}

Let $X$ be a del Pezzo threefold, i.e. $X$ is a smooth projective threefold such
that $-K_X =2 H$ where $H$ is an ample Cartier divisor on $X$. We denote by
$$
d := H^3
$$
the degree of the del Pezzo threefold. We also assume that $\pic(X) \cong \Z H$ (see Subsection \ref{subsectionproofdelPezzothreefold} for the other cases), so
by \cite[Thm.3.3.1]{EncV} one has $d \leq 5$.

A line on $X$ is a rational curve $l \subset X$ such that $H \cdot l=1$.
By \cite[Cor.3.1.11]{EncV} a general element $D \in |H|$ is smooth, by adjunction it is a del Pezzo surface of degree $d$. A $(-1)$-curve in $D$ satisfies $H \cdot l = -K_D \cdot l=1$, so it is a line. Since $D$ contains a $(-1)$-curve, we see that $X$ contains a line.

If $d \geq 2$, the linear system $|H|$ is basepoint free and $2H$ is very ample (see for instance \cite[Theorem 2.4.5]{EncV}) and every line is a smooth rational curve by \cite[Lemma 2.1.1]{KuznetsovProkhorovShramov2018}.

If $d=1$, the linear system $|H|$ has a unique base point $p$, and by \cite[Proposition 3.2]{Tik81} there exist
lines that are not smooth. The geometry of these lines can be understood as follows: let $x \in l_{\sing}$ be a singular point of the line. Then $x \neq p$: otherwise we can find a divisor $D \in |H|$ that does not contain $l$,
so $D \cdot l=H \cdot l=1$ gives a contradiction. Thus
we have $H^0(X, \sI_x \otimes \sO_X(H))=2$, and $H \cdot l=1$ implies that $l$ is contained in the base locus
of the pencil $|\sI_x \otimes \sO_X(H)|$. Since $H^3=1$
we obtain that $l$ coincides with the base locus, in particular it is a complete intersection of two fundamental divisors $D_1, D_2 \in |\sI_x \otimes \sO_X(H)|$.
By adjunction we see that $\omega_l \simeq \sO_l$,
so $l$ is isomorphic to a singular plane cubic (see also \cite[Remark 2.1.3]{KuznetsovProkhorovShramov2018}).

Let $\Sigma(X)$ be the Hilbert scheme of lines of $X$. Let $\Sigma_0$ be an irreducible component of $\Sigma(X)$ and consider the reduced scheme structure on $\Sigma_0$. Since $-K_X \cdot l=2$ we know by \cite[II,Thm.1.15]{Ko96} that $\Sigma_0$ has dimension at least two. Restricting to $\Sigma_0$ the universal family of lines, we obtain a diagram:
\[
\xymatrix{
   \mathcal{L}_0(X) \ar[d]_{q}\ar[r]^{e}  &  X\\
   \Sigma_0               &
}
\]

\begin{lemma} \label{lemmalinesfree}
	Let $X$ be a smooth del Pezzo threefold with $\rho(X)=1$. Then every irreducible component of $\Sigma(X)$ has dimension two and its general point corresponds to a free rational curve. In particular, every irreducible component of $\Sigma(X)$ is generically smooth. Moreover, the map $e\colon \mathcal{L}_0(X)\rightarrow X$ is surjective and generically finite.
\end{lemma}

The case $d \geq 3$ is well-known \cite[Prop.3.3.5]{EncV}, the case
$d=2$ is shown in \cite[Lemma 2.2.6]{KuznetsovProkhorovShramov2018}.

\begin{proof}
We can assume that $d=1$, we will follow the arguments of 
\cite{KuznetsovProkhorovShramov2018}.
Let $\Sigma_0$ be an irreducible component of $\Sigma(X)$. Assume that a line corresponding to a very general point of $\Sigma_0$ is not free, 
in particular $e$ is not surjective \cite[Prop.4.14]{Deb01}.
Since $k :=\dim(\Sigma_0)$ is at least two, we see that
$S=e(\mathcal{L}_0(X))$ is a surface.
Let $n\colon S'\rightarrow S$ be the normalisation. Then $\dim(\Sigma(S'))\geq 2$, where $\Sigma(S')$ is the Hilbert scheme of lines on $S'$ with respect to $n^*H$. Then we obtain $S\cong \PP^2$ by \cite[Thm.3.5, Thm.3.6]{Kebekus2002}. As a consequence, $n^*H$ is a hyperplane section of $\PP^2$; that is, $H^2\cdot S=1$. Thus $S$ is an element in $\vert H\vert$. Let $\mu\colon X\rightarrow \PP(1,1,1,2)$ be the double cover defined by $\vert 2H\vert$ (see Section \ref{Section:del-Pezzo-degree-one}). Then $\mu(S)\cong \PP(1,1,2)$ and $S\rightarrow \mu(S)$ is a double cover. In particular, the induced morphism $\mu'\colon \PP^2\rightarrow \PP(1,1,2)$ is a double cover with $(\mu')^*\sO_{\PP(1,1,2)}(2)\cong \sO_{\PP^2}(2)$. Let $i\colon \PP(1,1,2)\rightarrow \PP^4$ be the inclusion defined by $\vert \sO_{\PP(1,1,2)}(2)\vert$. Then we have the following commutative diagram:
	\[
	\xymatrix{
	    \PP^2\ar[r]^{\Phi}\ar[d]_{\mu'}&  \PP^5\ar@{-->}[d]^{\pi}\\
	  \PP(1,1,2)\ar[r]_{\quad i}  &  \PP^4
	}
	\]
	where $\Phi\colon \PP^2\rightarrow\PP^5$ is the Veronese embedding and $\pi$ is a projection from a point $p\in \PP^5\setminus \Phi(\PP^2)$. Denote by $V$ the image $\Phi(\PP^2)$. Then the restriction $\pi\vert_{V}\colon V\rightarrow \pi(V)$ is birational and $\pi(V)\subset \PP^4$ is a surface of degree four (see \cite[p.\,366, 10.5.5]{BeltramettiCarlettiGallaratiMontiBragadin2009} for details). Nevertheless, as $\mu'$ is of degree two by the construction, we obtain a contraction. Hence, a general line in $\Sigma_{0}$ is free. As a consequence, $\Sigma_{0}$ is two-dimensional and generically smooth (see \cite[Corollary 2.1.6]{KuznetsovProkhorovShramov2018}). 
\end{proof}

 By Lemma \ref{lemmalinesfree} every line on $X$ belongs to a an unsplit family $\mathcal{K}\subset \Chow{X}$ of minimal rational curves. Denote by $F(X)$ the closed subset of the Chow variety parametrising lines. By what precedes we know that $F(X)$ has pure dimension two, moreover $F(X)$ is irreducible if $3 \leq d \leq 5$ \cite[Proposition 3.5.6 and 3.5.8]{EncV}
or $1 \leq d \leq 2$ and $X$ is general in its deformation family
(\cite{Tik81} for $d=1$ and \cite{Tik80, Wel81} for $d=2$).

Let $\sK \subset F(X)$ be an irreducible component.
Let $\holom{q}{\sU}{\sK}$ be the normalisation of the universal family, and denote by
$\holom{e}{\sU}{X}$ the evaluation morphism. The evaluation morphism is generically finite, and we set
$$
k := \deg(e).
$$

Let $D \in |H|$ be a general divisor. We set
\begin{equation} \label{definer}
  r :=
\# \{
[l] \in \sK \ | \ l \subset D
\}.
\end{equation}

\begin{remark} \label{remarkidentifylines}
Let us recall that every $(-1)$-curves in $D$ is a line.
Moreover if $d \geq 2$, then, since $-K_D \cdot l=1$ and $l$ is smooth \cite[Lemma 2.1.1]{KuznetsovProkhorovShramov2018}, every line $l \subset D$ is a $(-1)$-curve.
\end{remark}

We come to the key lemma of this section:

\begin{lemma} \label{lemmaclassdelpezzo}
	Let $X$ be a smooth del Pezzo threefold of degree $d$ such that $\pic(X) \cong \Z H$.
	Denote by $\check{\mathcal{C}}$ the total dual VMRT associated to the unsplit
	family of rational curves $\sK$. Then one has
	\[
	[\check{\mathcal{C}}]=k \zeta+\pi^* \left(
	\frac{r}{d}-k
	\right)
	H.
	\]
\end{lemma}

\begin{proof}
Using the notation of Section \ref{subsectionVMRT}, we have $p=0$.
    By Remark \ref{remarkVMRTunsplit} we are in the situation of Proposition \ref{propositionclassVMRT}. Since the family of rational curves in unsplit, we have
    $\overline{\sU}=\sU$ and $e=\bar e$. Thus we only have to show that
    $$
    - e_* c_1(T_{\sU/\sK}) = \left(
	\frac{r}{d}-k
	\right)
	H.
    $$
    We will use an argument inspired by \cite[Sect.10]{CG72}: the pull-back
    $e^* \sO_X(H)$ is a nef and big line bundle on $\sU$ that has degree one
    on the $q$-fibres. Thus we have
    $$
    \sU \cong \PP(V)
    $$
    where $V := q_* (\sO_X(H))$ and $e^* H$ identifies to the tautological divisor
    on the projectivised bundle $\PP(V)$. Thus we have
    \begin{equation} \label{segredelpezzo}
    c_1^2(V)-c_2(V) = (e^* H)^3 = \deg(e) H^3 = k \cdot d.
    \end{equation}

    Let now $D \in |H|$ be a general element, and denote by $D' \subset \PP(V)$ its preimage. We claim that $D'$ is irreducible. For the proof of the claim we make a case distinction:

    {\em 1st case. Assume that $d \geq 2$, so $|H|$ is basepoint-free.}
    Since the linear system $e^* |H|$ is basepoint-free, the divisor $D'$ is smooth by Bertini's theorem. Since it is nef and big, it is connected.

        {\em 2nd case. Assume that $d=1$, so $|H|$ has a unique basepoint $p$.}
    Since the base locus of the linear system $e^* |H|$ is the preimage of the point $p$,
    and $\fibre{e}{p}$ has dimension at most one, we see that there are no fixed components. Moreover, since $e^* |H|$ is not composed with a pencil, a general member of the linear system is irreducible, so $D'$ is irreducible.

        The map $D' \rightarrow \sK$ is birational since $e^* H$ has degree one on the $q$-fibres.
    Moreover $[l] \in \sK$ is in the image of the exceptional locus of $D' \rightarrow \sK$
    if and only if the corresponding curve $l \subset X$ is contained in $D$. Thus, using the notation introduced above, the number of exceptional curves is equal to $r$.
    Since the divisor $D'$ can also be considered as a rational section of $\PP(V) \rightarrow \sK$ induced by a section $s \in H^0(\sK, V)$ that does not vanish in codimension one, we obtain that
    $$
    c_2(V) = r.
    $$
    Combined with \eqref{segredelpezzo} we obtain
    $$
    c_1^2(V) = k \cdot d + r.
    $$
    By the projection formula this implies
    $$
    (e_* q^* c_1(V)) \cdot H^2 = q^* c_1(V) \cdot e^* H^2 =  q^* c_1(V) \cdot
    c_1(\sO_{\PP(V)}(1))^2 = c_1^2(V) = k \cdot d + r.
    $$
    Since $\pic(X) \cong \Z H$, we have $e_* q^* c_1(V) = m H$ for some $m \in \N$. By the preceeding computation we have
    \begin{equation} \label{formulam}
       m d = k \cdot d + r \ \Leftrightarrow \ m = k + \frac{r}{d}.
    \end{equation}
    We are finally ready for the conclusion: since $\sU \cong \PP(V)$ we have
    $$
    c_1(T_{\sU/\sK}) = 2 c_1(\sO_{\PP(V)}(1)) - q^* c_1(V).
    $$
    Since $c_1(\sO_{\PP(V)}(1))=e^* H$ we obtain
    $$
    - e_*(c_1(T_{\sU/\sK})) = - 2 k H + e_* q^* c_1(V) = (- 2 k + m) H.
    $$
    Conclude with \eqref{formulam}.
    \end{proof}

As an immediate application, one can easily derive the total dual VMRT of del Pezzo threefolds
of low degree:

\begin{theorem}
	\label{thm:dual-VMRT-V5}
	Let $X$ be a smooth del Pezzo threefold of degree $d \leq 5$ such that $\pic(X) \cong \Z H$. If 	$F(X)$ is irreducible,
	the class $[\check{\mathcal{C}}]$ of the total dual VMRT of $X$ is given by the following table :

\begin{center}
\begin{tabular}{|c|c|c|c|c|c|c|c|c|c|c|c|}
\hline
$d$ & 1 & 2 & 3 & \hspace{0.5cm} 4 \hspace{0.5cm} & 5   \\ \hline
$[\check{\mathcal{C}}]$  & $60\zeta+m\pi^*H, \ m \geq 180$ & $12 \zeta + 16\pi^* H$
& $6\zeta + 3\pi^* H$ & $4\zeta$ & $3\zeta - \pi^* H$   \\ \hline
\end{tabular}
\label{tableclasses}
\end{center}

If $F(X)$ is reducible (in particular $d=1$ or $d=2$), denote by
$(\sK_i)_{i=1\ldots l}$ its irreducible components, and by
$[\check{\mathcal{C}_i}]$ the corresponding total dual VMRTs.
Then the following holds:

\begin{center}
\begin{tabular}{|c|c|c|c|c|c|c|c|c|c|c|c|}
\hline
$d$ & 1 & 2   \\ \hline
$\sum_{i=1}^l [\check{\mathcal{C}_i}]$  & $60\zeta+m\pi^*H, \ m \geq 180$ & $12 \zeta + 16\pi^* H$
   \\ \hline
\end{tabular}
\label{tableclasses2}
\end{center}
\end{theorem}

\begin{remark*}
    We do not know if there are del Pezzo threefolds such that $F(X)$ is reducible. However the irreducibility seems to be known only in the case where $X$ is general in its deformation family.
\end{remark*}

\begin{proof}
	We will use the notation of Lemma \ref{lemmaclassdelpezzo}. By the lemma the statement reduces to computing the numbers $r$ and $k$. 
	
	{\em 1st case. $F(X)$ is irreducible.} If $d \geq 2$, then
$r$ is the number of $(-1)$-curves in a general hyperplane section $D$ (see Remark \ref{remarkidentifylines}). 
If $d=1$, then
$r$ is at least the number of $(-1)$-curves in $D$. 

This number is known for every degree $d$, see \cite[Table 8.1]{Dol12}.

\begin{center}
\begin{tabular}{|c|c|c|c|c|c|c|c|c|c|c|c|}
\hline
$d$ & 1 & 2 & 3 & 4 & 5   \\ \hline
\mbox{number of $(-1)$-curves}  & 240 & 56 & 27 & 16 & 10    \\ \hline
\end{tabular}
\label{tableminusone}
\end{center}	
	
 We will compute $k$ case-by-case:
	
	If $d=1$, the manifold $X$ is a sextic in the weighted projective space $\PP(3,2,1,1,1)$. By \cite[Prop.6.7, Prop.4.7]{HK15} the VMRT is a complete intersection of degree $(4,5,6)$ in $V(U,Q)$, which is an open subset of $\PP(1,1,1,2)$ (see \cite[Notation 3.8]{HK15}). Thus we have $k=60$.
	
	If $d=2$, the manifold $X$ is a double cover of $\PP^3$ ramified along a quartic \cite[Thm.3.3.1]{EncV}. By \cite[Thm.1.1]{HK13}, the degree of the evaluation map is $k=12$.

	For $d=3$, the manifold $X$ is isomorphic to a smooth cubic hypersurface in $\PP^4$. By \cite{CG72} we  have $k=6$.
	
	For $d=4$, the manifold $X$ is isomorphic to a smooth complete intersection of two quadrics in $\PP^5$. Then we have $k=4$ \cite[Example 4.3]{CZ19}.
	
	For $d=5$, the manifold $X$ is isomorphic to a linear section of $\mbox{\rm Gr(2,5)}\subset \PP^9$. Since the VMRT of $\mbox{\rm Gr(2,5)}$ is isomorphic to the Segre embedding of $\PP^1\times \PP^2$ in $\PP^5$, the VMRT of $X$ is isomorphic to a linear section of $\PP^1\times \PP^2$. As a consequence, we have $k=3$ (see also for instance \cite[Lemma 2.3 (1)]{FN89}).

{\em 2nd case. $F(X)$ is reducible.} Denote by $k_i$ the degree of the evaluation map
corresponding the component $\sK_i$ and by $r_i$ the number defined in \eqref{definer}.
Since every $(-1)$-curve on $D$ belongs to one of the families $\sK_i$, we have
$\sum_{i=1}^l r_i=r$ where $r$ is at least the number given in Table \ref{tableminusone} (cf. Remark \ref{remarkidentifylines}).
Note also that $\sum_{i=1}^l k_i$ is the number of lines passing through a general point of $X$.
We claim that this number is an invariant in the deformation family. Assuming this for the time being, let us show how to conclude : by Lemma \ref{lemmaclassdelpezzo} we have
	$[\check{\mathcal{C}_i}]=k_i \zeta+\pi^* \left(
	\frac{r_i}{d}-k_i
	\right)
	H$.
Thus we have
$$
\sum_{i=1}^l [\check{\mathcal{C}_i}] = \sum_{i=1}^l k_i \zeta+\pi^* \left(
	\frac{r}{d}- \sum_{i=1}^lk_i
	\right) H.
$$
Since $\sum_{i=1}^l k_i$ is invariant in the deformation family, we know its value
from the case where $F(X)$ is irreducible.

{\em Proof of the claim.} If $\mathcal X \rightarrow \Delta$
is a smooth family of del Pezzo threefolds, and $f: \Delta \rightarrow X$ is a section
passing through a very general point of the central fibre $\mathcal X_0$, denote by
$\sV$ the union of the irreducible components of the space $\mbox{RatCurves}^n(f, \mathcal X/\Delta)$ \cite[II,(2.11.2)]{Ko96}
that parametrise lines. Since lines have minimal degree with respect to the polarisation,
$\sV$ is proper and finite over $\Delta$. Moreover, since a very general point is in the free locus of $\mathcal X_0$, it follows from \cite[II, Thm.1.7]{Ko96} that
$\sV \rightarrow \Delta$ is smooth. Thus the number of preimages of $\sV \rightarrow \Delta$ does not depend on $t \in \Delta$.
\end{proof}

\begin{remark}
	One can also use the detailed description of the universal family of lines on $X$ to get the same result: for $d=3$, by \cite{CG72}, we have $V\cong \Omega_{F(X)}$ and $K^2_{F(X)}=45$; for $d=5$, by \cite{FN89}, $F(X)$ is isomorphic to $\PP^2$ and $\det(V)\cong \sO_{\PP^2}(5)$.
\end{remark}

\begin{corollary} \label{corollarydelpezzonotbig}
    Let $X$ be a smooth del Pezzo threefold of degree $d \leq 2$ such that $\pic(X) \cong \Z H$.
    Then $T_X$ is not big.
\end{corollary}

\begin{proof}
    If $F(X)$ is irreducible, we simply apply the first part of Theorem \ref{thm:dual-VMRT-V5}, and Corollary \ref{corollarytangentnotbig}.

    If $F(X)$ is reducible, we know by the second part of Theorem \ref{thm:dual-VMRT-V5}
    that we can choose one irreducible component $\sK_i \subset F(X)$ such that
    the corresponding total dual VMRT $\check{\mathcal{C}_i}$
    has class $[\check{\mathcal{C}_i}]= e \zeta + m \pi^* H$ for some $m \geq 0$. Thus
    Corollary \ref{corollarytangentnotbig} still applies.
\end{proof}

\subsection{Projective geometry of del Pezzo threefolds}

The general technique developed in the previous subsection allows us to decide whether $T_X$ is big or not.  The pseudoeffectivity of $T_X$ in the cases $d=1$ and $d=2$ will
require a more detailed treatment. We start with a technical lemma:

\begin{lemma}
	\label{lemma:del-Pezzo-intersection-number}
	Let $X$ be a smooth del Pezzo threefold of degree $d$ such that $\pic(X)\cong \Z H$. Then we have
	\begin{equation}
	\label{equation:schur-function-threefolds}
	\zeta^5=8d-44-b_3(X),\quad \zeta^4\cdot \pi^*H=4d-12,\quad \zeta^3\cdot \pi^*(H^2)=2d.
	\end{equation}
	
\end{lemma}

\begin{proof}
	Since $X$ is a Fano manifold, by Kodaira vanishing, we have $h^i(X,\sO_X)=0$ for $i\geq 1$. Thus, by our assumption and using the Hodge decomposition, we get $b_1(X)=0$ and $b_2(X)=\rho(X)=1$. On the other hand, it is well-known that the degree of the top Chern class is equal to the topological Euler characteristic, i.e.,
	\[
	c_3(X)=\chi_{\rm top}(X)=4-b_3(X).
	\]
	Moreover by Riemann-Roch in dimension three one has
	$$
	24 \chi(X, \sO_X) = -K_X \cdot c_2(X) = 2 H \cdot c_2(X).
	$$
	Thus we have $H \cdot c_2(X)=12$.
    The result now follows from the formulas for Segre classes \cite[Ex.8.3.4, Ex.8.3.5]{Laz04b}.
\end{proof}

\subsubsection{del Pezzo threefolds of degree one}

\label{Section:del-Pezzo-degree-one}

Let $X$ be a smooth del Pezzo threefold of degree one. Then $X$ is a double cover
$$
\mu : X \rightarrow W
$$
of the Veronese cone $W=\PP(1,1,1,2)$. The branch locus of $\mu$ consists
of the vertex of $W$ and a smooth weighted hypersurface $S$ of degree $6$ in $W$ not passing through the vertex of $W$ (see for instance \cite[Theorem 2.4.5]{EncV}). 

Denote by $A$ a Weil $\Q$-Cartier divisor on $W$ associated to the reflexive sheaf $\sO_W(1)$. Then we have $\mu^*A=H$,
and the unique base point $p \in X$ of $\vert H\vert$ maps onto the vertex of $W$, i.e. the unique base point of $|A|$.
Denote by $F$ the fibre of $\holom{\pi}{\PP(T_X)}{X}$ over the base point $p$.

	Denote by $W^\circ=W\setminus\{\mu(p)\}$ the smooth locus of $W$,
	and denote $X\setminus\{p\}$ by $X^\circ$.
Since the ramification divisor of the double cover is isomorphic to $S$, we will identify $\mu^{-1}(S)$ with $S$.
The restriction of the cotangent map composed with the canonical map
$\Omega_{X^\circ}\vert_S\rightarrow \Omega_S$ gives a surjective map
$$
\mu^*\Omega_{W^\circ}\vert_S\rightarrow \Omega_{X^\circ}\vert_S\rightarrow \Omega_S.
$$
Thus the second exterior power of the cotangent map
	$$
	\mu^* (\wedge^2 \Omega_{W^\circ}) \rightarrow \wedge^2 \Omega_{X^\circ}
	$$
	has degree one along $S$ and its image identifies to $K_S$.
	
The kernel of $\Omega_{X}\vert_S\rightarrow \Omega_S$ is isomorphic to $\sO_S(-S)$, so
we obtain a canonical splitting
	\[
	\Omega_X\vert_S = \Omega_S\oplus \sO_S(-S)\cong \Omega_S\oplus \sO_S(-3H).
	\]
 Using this
 canonical splitting, we have a canonical quotient $T_X|_S \rightarrow T_S$
 which allows us to consider $\PP(T_S)$ as a subvariety of $\PP(T_X|_S)$.

\begin{proposition}
	\label{prop:nef-cone-del-Pezzo-degree-one}
	Let $X$ be a smooth del Pezzo threefold of degree one.
	Then the following statements hold:
	\begin{enumerate}[(a)]
			\item there exists an irreducible divisor $D\in \vert \zeta+\pi^*H\vert$;
		\item the base locus of $\vert\zeta+2\pi^*H\vert$ is contained in $\PP(T_S)\cup F$;
		\item the base locus of $\vert(\zeta+3\pi^*H)\vert_{\PP(T_S)}\vert$ has dimension at most one. In particular, the restriction of $(\zeta+3\pi^*H)^2$ to any codimension two subvariety is pseudoeffective;
		\item $\zeta+4 \pi^*H$ is nef;
		\item if $X$ is general in its deformation family, then $\zeta+ \varepsilon \pi^*H$ is nef if and only if $\varepsilon \geq 3$.
	\end{enumerate}
\end{proposition}

\begin{remark}
    We have $b_3(X)= 42$ by \cite[Table 12.2]{EncV}. Thus by Lemma \ref{lemma:del-Pezzo-intersection-number} we have
\begin{equation} \label{equationnumbersdegreeone}
\zeta^5=-78,\quad \zeta^4\cdot \pi^*H=-8,\quad \zeta^3\cdot \pi^*(H^2)=2
\end{equation}
\end{remark}

\begin{proof}

	In what follows we will frequently use that, since $X$ is smooth
	and $\codim(X\setminus X^\circ)=3$, one has
	\[
	H^0(X^\circ, T_{X^\circ}(mH)) \cong H^0(X, T_X(mH)) \qquad \forall \ m \in \Z.
	\]

We consider the cotangent map $\mu^* \Omega_{W^\circ}\rightarrow \Omega_{X^\circ}$.
Since $-K_W = 5 A$ by \cite{BR86} and $\mu^* A=H$, taking the second exterior power yields an injective morphism
\begin{equation} \label{tangentmapdegreeone}
\mu^* T_{W^\circ} \cong \mu^* (\wedge^2 \Omega_{W^\circ} (5A))
\rightarrow \wedge^2 \Omega_{X^\circ} (5H) \cong T_{X^\circ} (3H)
\end{equation}
that is an isomorphism in the complement of $S$. Moreover, by the discussion preceding Proposition \ref{prop:nef-cone-del-Pezzo-degree-one}
the restriction
\begin{equation} \label{tangentmapdegreeonebis}
	(\mu^* T_W)\vert_S \rightarrow (T_X(3H))
	\vert_S = K_S(5H) \oplus \Omega_S(2H)
\end{equation}
has image $K_S(5H)$.

We have the following weighted Euler sequence
	\begin{equation}
	\label{equation:weighted-Euler-sequence}
	0\rightarrow \sO_{W^\circ}\rightarrow \sO_{W^\circ}(1)^{\oplus 3}\oplus \sO_{W^\circ}(2)\rightarrow T_{W^\circ}\rightarrow 0.		
	\end{equation}

{\em Proof of (a)}
	Twisting \eqref{equation:weighted-Euler-sequence} with $\sO_{W^\circ}(-2)$ yields a non-zero map
	\[
	\alpha\colon\sO_{W^\circ}\rightarrow T_{W^\circ}(-2).
	\]
	Moreover, from the construction of the Euler sequence, one can easily see that the map $\alpha$ is of rank one everywhere. The composition of $\alpha$ with the morphism \eqref{tangentmapdegreeone}, twisted with $-2H$, gives a morphism
	\[
	\widetilde{\alpha}\colon \sO_{X^\circ}\rightarrow \mu^*(T_{W^\circ}(-2))\rightarrow T_{X^\circ}(H)
	\]
	that has rank one in the complement of $S$.
 Let $D\in \vert \zeta+\pi^*H\vert$ be the element corresponding to the section $\widetilde{\alpha}$ (extended to $X$). Since $\tilde \alpha$ has rank one in the complement of $S \cup p$,
 and $F=\fibre{\pi}{p}$ has codimension three in $\PP(T_X)$, the divisor
 $D$ is irreducible in the complement of $\PP(T_X\vert_S)$. Thus, if $D$ is reducible,
 then it contains the divisor $\PP(T_X\vert_S)$.
 However, note that the class of $[\PP(T_X\vert_S)]$ in $\pic(\PP(T_X))$ is $3 \pi^*H$.
 Thus if $D$ is reducible, the class $\zeta-2 \pi^* H$ is effectively represented.
 Yet by Lemma \ref{lemmasubstractbig} this implies that $\zeta$ is big, in contradiction to Corollary \ref{corollarytangentnotbig}. Hence the divisor $D$ is irreducible.	

{\em Proof of (b)} Since $\sO_{W^\circ}(1)$ is globally generated on $W^\circ$, the
weighted Euler sequence \eqref{equation:weighted-Euler-sequence} twisted with
$\sO_{W^\circ}(-1)$ shows that $T_{W^\circ}(-1)$ is globally generated.
	Then the inclusion \eqref{tangentmapdegreeone}, twisted with $-H$ shows
 that $T_{X}(2H)$ is globally generated outside $S\cup\{p\}$.
 Moreover, twisting \eqref{tangentmapdegreeonebis} with $\sO_{W^\circ}(-1)$ shows that
 the global sections generate the canonical supplement of
 $\Omega_S(H)$ in $(T_{X}(2H))|_S$.
Thus the base locus of  $\zeta+2\pi^*H$ is contained in the union of $F$
and $\PP(T_S) \subset \PP(T_X|_S)$.

{\em Proof of (c)}
Since $S$ is a smooth weighted hypersurface of degree six in $W$ that does not contain the vertex, the natural morphism $\gamma\colon S\rightarrow\PP^2$ induced by the projection $W\dashrightarrow\PP^2$ from the vertex is a finite morphism of degree three. By Lemma \ref{lemma:freeness-twist-tangent-bundle}, the vector bundle $\Omega_{\PP^2}(2)$ is globally generated. The cotangent map $\gamma^*\Omega_{\PP^2}\rightarrow \Omega_S$
twisted with
$2H$ shows that $\Omega_S(2H)$
is globally generated in the complement of the ramification locus $C$ of $\gamma$.
Moreover, since the cotangent map has rank one in a general point of $C$,
the global sections generate a subsheaf of rank at least one along $C$.

Denote by $\holom{\pi_S}{\PP(T_S)}{S}$ the projectivisation, and by $\zeta_S \rightarrow \PP(T_S)$ the tautological class. Since $T_S\cong \Omega_S(-H)$, the preceding paragraph shows that $T_S(3H)$ is globally generated in the complement of $C$ and the base locus of $\zeta_S + \pi_S^*(3H)$ does not contain any divisor. Since the restriction of $\zeta$
to $\PP(T_S)$ is $\zeta_S$, this shows the first statement.

For the proof of the second statement, let $\Lambda\subset \PP(T_X)$ be an arbitrary subvariety of codimension two. By our argument above, we may assume that $\Lambda\not=\PP(T_S)$. Since the base locus of $\vert \zeta+2\pi^*H\vert$ is contained in $\PP(T_S)\cup F$ by (b), it follows that there exists an effective element $\Gamma\in \vert (\zeta+2\pi^*H)\vert_{\Lambda}\vert$. Let $\Gamma=\sum a_i\Gamma_i$ be the decomposition with $\Gamma_i$ irreducible, reduced and pairwise distinct $2$-cycles. If $\Gamma_i$ is not contained in $\PP(T_S)\cup F$, then $(\zeta+2\pi^*H)\vert_{\Gamma_i}$ is pseudoeffective by (b). If $\Gamma_i$ is contained in $\PP(T_S)$, then the restriction $(\zeta+3\pi^*H)\vert_{\Gamma_i}$ is pseudoeffective as the base locus of $\vert(\zeta+3\pi^*H)\vert_{\PP(T_S)}\vert$ has dimension at most one. If $\Gamma_i$ is contained in $F$, then $\Gamma_i=F$ and therefore $\zeta\vert_F=c_1(\sO_{\PP^2}(1))$ is ample. Hence we obtain that $((\zeta+2\pi^*H) \cdot (\zeta+3\pi^*H))\vert_{\Lambda}$ is pseudoeffective.

{\em Proof of (d)} By (b) and (c) we know that $\zeta+4 \pi^* H$ is nef
if its restriction to $\PP(T_S)$ is nef.
Yet $T_S(4H) \cong \Omega_S(3H) \cong \Omega_S(3A)$, so it is sufficient to show that
$\Omega_S(3A)$ is globally generated. Yet this is clear since
$$
\Omega_{W^\circ}(3A) \cong \wedge^2 (T_{W^\circ}(-1))
$$
is globally generated.

{\em Proof of (e)}	
	Finally assume that $X$ is general in its deformation family. By \cite[Proposition 2.1 and (1.2.7)]{Tik81}, the curve $C$ is smooth, irreducible and $\sO_S(C)\cong \sO_S(4H)$.
	By b) the non-nef locus of $\zeta+3\pi^*H$ is contained in $\PP(T_S\vert_C):=\Pi_C$. By \cite[Proposition 5.1]{Tik81}, there exists an irreducible smooth divisor $D$ on $\PP(\Omega_S\vert_C)$ such that
	\[
	\sO_{\PP(\Omega_S\vert_C)}(D)\cong \sO_{\PP(\Omega_S\vert_C)}(1)\otimes \overline{\pi}^*\sO_C(-3H)
	\]
	where $\overline{\pi}\colon \PP(\Omega_S\vert_C)\rightarrow C$ is the natural projection. Using the duality $\Omega_S\cong T_S(H)$, the class of $D$ on $\Pi_C=\PP(T_S\vert_C)$ is $(\zeta-2\pi^*H)\vert_{\Pi_C}$.	In particular, the non-nef locus of $(\zeta-2\pi^*H)\vert_{\Pi_C}$ is contained in $D$. On the other hand, on $\PP(T_X)$, note that we have
	\[
    [\PP(T_S)]=3H\cdot (\zeta-3\pi^*H), \quad [\Pi_C]=12H^2\cdot (\zeta-3\pi^*H).
	\]
	Then an easy computation using \eqref{equationnumbersdegreeone}
	shows that the intersection number $(\zeta+3\pi^*H)\cdot D=0$. Hence, $\zeta+ \varepsilon \pi^*H$ is nef if and only if $\varepsilon\geq 3$.
\end{proof}

\begin{corollary} \label{cor:pseudoeffectivity-degree-one}
	Let $X$ be a smooth del Pezzo threefold of degree one. Then $T_X$ is not pseudoeffective.
\end{corollary}

\begin{proof}
Arguing by contradiction we assume that $\zeta$ is pseudoeffective.
By Corollary \ref{corollarytangentnotbig} we know that $\zeta$ is not big, so it generates an extremal ray
in the pseudoeffective cone of $\PP(T_X)$. By Corollary \ref{corollaryextremalclass}
this implies that either $\zeta$ is nef in codimension one or represented by
an irreducible divisor $D_1$.

By Proposition \ref{prop:nef-cone-del-Pezzo-degree-one},a) there exists
an irreducible divisor in $D_2 \in \vert \zeta+\pi^*H\vert$.
Since $\zeta$ is nef in codimension one or represented by
an irreducible divisor $D_1 \neq D_2$,
the intersection product $D_1\cdot D_2$ is a pseudoeffective codimension two cycle. Since, by Proposition \ref{prop:nef-cone-del-Pezzo-degree-one},c) the restriction of $(\zeta+3\pi^*H)^2$ to any codimension two cycle is pseudoeffective and,
by Proposition \ref{prop:nef-cone-del-Pezzo-degree-one},d), the divisor $\zeta+4\pi^*H$ is nef, we obtain
	\[
	D_1\cdot D_2\cdot (\zeta+3\pi^*H)^2\cdot (\zeta+4\pi^*H)\geq 0.
	\]
Yet a straightforward computation using \eqref{equationnumbersdegreeone} shows
that the intersection product is $-11$, a contradiction.
\end{proof}

\subsubsection{del Pezzo threefolds of degree two}

Let $S\subset \PP^3$ be a smooth quartic surface. Then $S$ is a K3 surface and we have a canonical isomorphism $T_S\cong \Omega_S$. 
Recall that by the Noether-Lefschetz Theorem, $S$ has Picard number one if it is very general. If $S$ has Picard number one, Gounelas and Ottem have investigated the positivity of $\Omega_S$ in \cite[Section 4.2]{GO18}.  Denote by $U\subset \PP(T_S)$ the incidence variety associated to the surface of bitangents of $S$ (cf. \cite[p.30]{Wel81}).  Denote by $H$ the hyperplane section of $S$ and by $\pi\colon \PP(T_S)\rightarrow S$ the natural projection.
In the following lemma, we collect some results related to the positivity of $T_S$. 

\begin{lemma}
	\label{lemma:quartic-surface}
	Let $S\subset \PP^3$ be a smooth quartic surface. Then the following statements hold.
	\begin{enumerate}[(a)]
		\item The class $[U]$ in $\pic(T_S)$ is $6\zeta+8\pi^*H$. In particular, $\zeta+\frac{4}{3}\pi^*H$ is pseudoeffective.
		
		\item If $S$ is very general in its deformation family, then $\zeta+\varepsilon\pi^*H$ is pseudoeffective if and only if $\varepsilon\geq \frac{4}{3}$.
		
		\item $\zeta+2\pi^*H$ is nef.
		
		\item If $S$ is general in its deformation family, then $\zeta+\frac{3}{2}\pi^*H$ is nef.
	\end{enumerate}
\end{lemma}

\begin{proof}
	Statement (a) is proved in \cite[Proposition 2.3]{Tik80}, \cite[Prop.3.14]{Wel81} and (b) follows from \cite[Corollary 4.2]{GO18}. Statement (c) is a consequence of Lemma \ref{lemma:freeness-twist-tangent-bundle}. For (d), this is already proved in \cite[Section 4.2.2]{GO18} for $S$ being very general. On the other hand, note that $\zeta+\frac{3}{2}\pi^*H$ is nef if $\zeta+(\frac{3}{2}+\varepsilon)\pi^*H$ is ample for any rational numbers $0<\varepsilon\ll 1$. Since ampleness is an open property, it follows that $\zeta+\frac{3}{2}\pi^*H$ is nef for $S$ general in the deformation family.
\end{proof}

Let $X$ be a smooth del Pezzo threefold of degree two. Then $X$ is a double cover $\mu : X \rightarrow \PP^3$ such that $\mu^*\sO_{\PP^3}(1)\cong \sO_X(H)$ and the branch locus of $\mu$ is a smooth quartic surface $S\subset \PP^3$, see \cite[Theorem 3.3.5]{EncV}. Since the ramification divisor of the double cover is isomorphic to $S$, we will identify $\mu^{-1}(S)$ with $S$.
The restriction of the cotangent map composed with the canonical map
$\Omega_{X}\vert_S\rightarrow \Omega_S$ gives a surjective map
$$
\mu^*\Omega_{\PP^3}\vert_S\rightarrow \Omega_{X}\vert_S\rightarrow \Omega_S.
$$
Thus the second exterior power of the cotangent map
$$
\mu^* (\wedge^2 \Omega_{\PP^3}) \rightarrow \wedge^2 \Omega_{X}
$$
has degree one along $S$ and its image identifies to $K_S$.

The kernel of $\Omega_{X}\vert_S\rightarrow \Omega_S$ is isomorphic to $\sO_S(-S)$, so
we obtain a canonical splitting
\[
\Omega_X\vert_S = \Omega_S\oplus \sO_S(-S)\cong \Omega_S\oplus \sO_S(-2H).
\]
Using this
canonical splitting, we have a canonical quotient $T_X|_S \rightarrow T_S$
which allows us to consider $\PP(T_S)$ as a subvariety of $\PP(T_X|_S)$.

\begin{proposition}
	\label{prop:nef-cone-del-Pezzo-degree-two}
	Let $X$ be a smooth del Pezzo threefold of degree two. Then the following statements hold:
	\begin{enumerate}[(a)]
		\item the base locus of $\vert \zeta+\pi^*H\vert$ is contained in $\PP(T_S)$;
		
		\item the restriction $\left(\zeta+\frac{4}{3}\pi^*H\right)\vert_{\PP(T_S)}$ is pseudoeffective;
		
		\item $\zeta+2\pi^*H$ is nef;
		
		\item if $X$ is general in its deformation family, then $\zeta+\frac{3}{2}\pi^*H$ is nef.
	\end{enumerate}
\end{proposition}

\begin{proof}
	We consider the cotangent map $\mu^*\Omega_{\PP^2}\rightarrow \Omega_X$. Since $(\wedge^{2}\Omega_{\PP^3})(3)$ is globally generated (see Lemma \ref{lemma:freeness-twist-tangent-bundle}), we see that $(\wedge^2\Omega_X)(3H)\cong T_X(H)$ is globally generated in the complement of the ramification divisor $S$. By the paragraph before the proposition, we have
	\[
	(T_X(H))\vert_S\cong (\wedge^2\Omega_X)(3H)\vert_S \cong \Omega_S(H)\oplus \sO_S(3H).
	\]
	We see that the base locus of $\vert\zeta+\pi^*H\vert$ is contained in $\PP(T_S)\subset \PP(T_X\vert_S)$. Now the remaining statements follow easily from Lemma \ref{lemma:quartic-surface}.
\end{proof}

\begin{corollary}
	\label{prop:pseudoeffectiveity-degree-two}
	Let $X$ be a smooth del Pezzo threefold of degree two. If $T_X$ is pseudoeffective, then there exists an effective divisor $D\subset \PP(T_X)$ such that $D\sim m\zeta$ for some positive integer $m$. In particular, if $X$ is general in its deformation family, then $T_X$ is not pseudoeffective.
\end{corollary}

\begin{remark}
	We have $b_3(X)= 20$ by \cite[Table 12.2]{EncV}. Thus by Lemma \ref{lemma:del-Pezzo-intersection-number} we have
	\begin{equation} \label{equationnumbersdegreetwo}
	\zeta^5=-48,\quad \zeta^4\cdot \pi^*H=-4,\quad \zeta^3\cdot \pi^*(H^2)=4.
	\end{equation}
\end{remark}

\begin{proof}
	By Corollary \ref{corollarytangentnotbig} we know that $\zeta$ is not big, so it generates an extremal ray
	in the pseudoeffective cone of $\PP(T_X)$.
	If $\zeta$
	is modified nef, then $\zeta^2$ is a pseudoeffective cycle of codimension two
	and, since $\zeta+2\pi^*H$ is nef by Proposition \ref{prop:nef-cone-del-Pezzo-degree-two} (c), one has
	\[
	\zeta^2\cdot \left(\zeta+2\pi^*H\right)^3\geq 0.
	\]
	An elementary computation using \eqref{equationnumbersdegreetwo} shows that this intersection number is $-8$, a contradiction. Hence, by Corollary \ref{corollaryextremalclass}  there exists
	a unique prime divisor $D \subset \PP(T_X)$ generating the extremal ray
	$\R^+ \zeta$.
	
	Now we assume that $X$ is general in its deformation family. We claim that  the class
	$\zeta \cdot (\zeta+\pi^* H) \cdot (\zeta+\frac{4}{3} \pi^* H)$ is
	a pseudoeffective cycle of codimension three. Since $\zeta+\frac{3}{2} \pi^* H$ is nef by Proposition \ref{prop:nef-cone-del-Pezzo-degree-two} (d), this implies that
	$$
	\zeta \cdot (\zeta+\pi^* H) \cdot (\zeta+\frac{4}{3} \pi^* H) \cdot (\zeta+\frac{3}{2} \pi^* H)^2 \geq 0.
	$$
	An elementary computation using using \eqref{equationnumbersdegreetwo} shows that this intersection product is $\frac{-49}{6}$. We have reached a contradiction.
	
	{\em Proof of the claim.}
	Let $D \in |m \zeta|$ be the unique prime divisor generating the extremal ray
	$\R^+ \zeta$.
	If $D'$ is a general element of $\zeta+\pi^* H$, we have
	$$
	D \cdot D' = a [\PP(T_S)] + \sum_i Z_i
	$$
	where $a \geq 0$ and the $Z_i$ are mobile cycles of codimension two.
	
	Let $\check{\mathcal{C}}$ be the dual VMRT. By Theorem \ref{thm:dual-VMRT-V5}
	the class of $\check{\mathcal{C}}$ is a positive multiple of $\zeta+\frac{4}{3} \pi^* H$. Since the cycles $Z_i$ are mobile, the intersection
	$\check{\mathcal{C}} \cdot Z_i$ is an effective cycle of codimension three.
	Moreover, by \cite[Cor.4.2]{GO18} (cf. Proposition \ref{prop:nef-cone-del-Pezzo-degree-two} (b)), the restriction of $\zeta+\frac{4}{3} \pi^* H$
	to $\PP(T_S)$ is a pseudoeffective divisor class. Thus
	$\check{\mathcal{C}} \cdot [\PP(T_S)]$ can be represented by an effective cycle of codimension three. This shows that the intersection
	$$
	m \zeta \cdot (\zeta+ \pi^* H) \cdot (\zeta+\frac{4}{3} \pi^* H) =
	D \cdot D' \cdot \check{\mathcal{C}}
	$$
	can be represented by an effective cycle of codimension three.
\end{proof}


\begin{remark} \label{remarkspecialcases}
    If the del Pezzo threefold is not general, the proof above does not work. Indeed assume that the quartic $S$ contains a line $l$. Then the quotient $T_S|_l \rightarrow \sO_{\PP^1}(-2)$ defines a curve $\tilde l$ such that $(\zeta+\frac{4}{3} \pi^* H) \cdot \tilde l<0$. Thus $\tilde l$ is contained in the surface $U$ and
    the restriction of $(\zeta+\frac{3}{2} \pi^* H)$ to $U$ is not nef. 
\end{remark}

\subsection{Proof of the main statement} \label{subsectionproofdelPezzothreefold}

We start by settling the cases that are not covered by our general setup:

\begin{proposition}
	\label{prop:del-Pezzo-degree>5}
	Let $X$ be a del Pezzo manifold of dimension $n$ and degree $d$. If $n\geq 3$ and $d\geq 6$, then $T_X$ is big.
\end{proposition}

\begin{proof}
	If $n\geq 4$, then $X$ is isomorphic to $\PP^2\times \PP^2$ (see \cite[Table 12.1]{EncV}). In particular, $X$ is toric and therefore $T_X$ is big by \cite[Corollary 1.3]{Hsi15}.
	
	If $n=3$, the only smooth del Pezzo threefolds with degree $d\geq 6$ are $\PP(T_{\PP^2})$, $\PP^1\times\PP^1\times\PP^1$ and $\PP_{\PP^2}(\sO_{\PP^2}\oplus\sO_{\PP^2}(1))$ \cite[Thm.3.3.1]{EncV}. For $X\cong \PP(T_{\PP^2})$, by \cite[Example 2]{SW04}, $T_X$ is big and $1$-ample. The last two are toric varieties, hence $T_X$ is big by \cite[Corollary 1.3]{Hsi15}.
\end{proof}

\begin{proof}[Proof of Theorem \ref{THM:bigness-pseudo-effectivity-del-Pezzo-threefolds}]
	For $d\leq 3$, this is done in Corollary \ref{cor:pseudoeffectivity-degree-one}, Corollary \ref{prop:pseudoeffectiveity-degree-two} and Theorem \ref{thm:Tangent-hypersurface-not-pseudo-effective}. For $d\geq 5$, it follows from Theorem \ref{thm:dual-VMRT-V5} and Proposition \ref{prop:del-Pezzo-degree>5}.
	
	For $d=4$, by Theorem \ref{thm:dual-VMRT-V5}, the total dual VMRT $[\check{\mathcal{C}}]$ has class $4\zeta$. In particular, $T_X$ is pseudoeffective, but it is not big by Corollary \ref{corollarytangentnotbig}.
\end{proof}


\begin{thebibliography}{BHPVdV04}

\bibitem[BCGM09]{BeltramettiCarlettiGallaratiMontiBragadin2009}
Mauro~C. {Beltrametti}, Ettore {Carletti}, Dionisio {Gallarati}, and Giacomo
  {Monti Bragadin}.
\newblock {\em {Lectures on curves, surfaces and projective varieties. A
  classical view of algebraic geometry. Transl. from the Italian by Francis
  Sullivan.}}
\newblock Z\"urich: European Mathematical Society (EMS), 2009.

\bibitem[BDO08]{BogomolovDeOliveira2008}
Fedor Bogomolov and Bruno De~Oliveira.
\newblock Symmetric tensors and geometry of {$\Bbb P^N$} subvarieties.
\newblock {\em Geom. Funct. Anal.}, 18(3):637--656, 2008.
\newblock doi: 10.1007/s00039-008-0666-7.

\bibitem[BHPVdV04]{BHPV04}
Wolf~P. Barth, Klaus Hulek, Chris A.~M. Peters, and Antonius Van~de Ven.
\newblock {\em Compact complex surfaces}, volume~4 of {\em Ergebnisse der
  Mathematik und ihrer Grenzgebiete. 3. Folge.}
\newblock Springer-Verlag, Berlin, second edition, 2004.
\newblock doi: 10.1007/978-3-642-57739-0.

\bibitem[Bou04]{Bou04}
S{\'e}bastien Boucksom.
\newblock Divisorial {Z}ariski decompositions on compact complex manifolds.
\newblock {\em Ann. Sci. \'Ecole Norm. Sup. (4)}, 37(1):45--76, 2004.
\newblock doi: 10.1016/j.ansens.2003.04.002.

\bibitem[BR86]{BR86}
Mauro Beltrametti and Lorenzo Robbiano.
\newblock Introduction to the theory of weighted projective spaces.
\newblock {\em Exposition. Math.}, 4(2):111--162, 1986.

\bibitem[BR90]{BrueckmannRackwitz1990}
Peter Br\"uckmann and Hans-Georg Rackwitz.
\newblock {\(T\)-symmetrical tensor forms on complete intersections.}
\newblock {\em {Math. Ann.}}, 288(4):627--635, 1990.
\newblock doi: 10.1007/BF01444555.

\bibitem[CG72]{CG72}
C.~Herbert Clemens and Phillip~A. Griffiths.
\newblock The intermediate {J}acobian of the cubic threefold.
\newblock {\em Ann. of Math. (2)}, 95:281--356, 1972.

\bibitem[CZ19]{CZ19}
Ciro Ciliberto and Mikhail Zaidenberg.
\newblock On {F}ano schemes of complete intersections.
\newblock {\em arXiv preprint}, 1903.11294., 2019.

\bibitem[Deb01]{Deb01}
Olivier Debarre.
\newblock {\em Higher-dimensional algebraic geometry}.
\newblock Universitext. Springer-Verlag, New York, 2001.
\newblock doi: 10.1007/978-1-4757-5406-3.

\bibitem[Dem77]{Dem76}
Michel Demazure.
\newblock Surfaces de del pezzo : V - mod\`eles anticanoniques.
\newblock {\em S\'eminaire sur les singularit\'es des surfaces}, 1976-1977.
\newblock talk:7.

\bibitem[Dol12]{Dol12}
Igor~V. Dolgachev.
\newblock {\em Classical algebraic geometry}.
\newblock Cambridge University Press, Cambridge, 2012.
\newblock A modern view.

\bibitem[Dru18]{Druel2018}
St\'{e}phane Druel.
\newblock A decomposition theorem for singular spaces with trivial canonical
  class of dimension at most five.
\newblock {\em Invent. Math.}, 211(1):245--296, 2018.
\newblock doi: 10.1007/s00222-017-0748-y.

\bibitem[EH16]{EH16}
David Eisenbud and Joe Harris.
\newblock {\em 3264 and all that---a second course in algebraic geometry}.
\newblock Cambridge University Press, Cambridge, 2016.
\newblock doi: 10.1017/CBO9781139062046.

\bibitem[FH91]{FH91}
William Fulton and Joe Harris.
\newblock {\em Representation theory}, volume 129 of {\em Graduate Texts in
  Mathematics}.
\newblock Springer-Verlag, New York, 1991.
\newblock A first course, Readings in Mathematics.

\bibitem[FN89]{FN89}
Mikio Furushima and Noboru Nakayama.
\newblock The family of lines on the {F}ano threefold {$V_5$}.
\newblock {\em Nagoya Math. J.}, 116:111--122, 1989.
\newblock doi: 10.1017/S0027763000001719.

\bibitem[GO18]{GO18}
Frank Gounelas and John~Christian Ottem.
\newblock Remarks on the positivity of the cotangent bundle of a {K3} surface.
\newblock {\em arXiv preprint arXiv:1806.09598}, 2018.

\bibitem[GW19]{GrebWong2019}
Daniel Greb and Michael~Lennox Wong.
\newblock Canonical complex extensions of k{\"a}hler manifolds.
\newblock {\em Journal of the London Mathematical Society}, to appear, 2019.
\newblock doi: 10.1112/jlms.12287.

\bibitem[Har77]{Har77}
Robin Hartshorne.
\newblock {\em Algebraic geometry}.
\newblock Springer-Verlag, New York, 1977.
\newblock Graduate Texts in Mathematics, No. 52.

\bibitem[HIM19]{HIM19}
Genki Hosono, Masataka Iwai, and Shin-Ichi Matsumura.
\newblock On projective manifolds with pseudo-effective tangent bundle.
\newblock {\em ArXiv preprint}, 1908.06421, 2019.

\bibitem[HK13]{HK13}
Jun-Muk Hwang and Hosung Kim.
\newblock Varieties of minimal rational tangents on double covers of projective
  space.
\newblock {\em Math. Z.}, 275(1-2):109--125, 2013.
\newblock doi: 10.1007/s00209-012-1125-6.

\bibitem[HK15]{HK15}
Jun-Muk Hwang and Hosung Kim.
\newblock Varieties of minimal rational tangents on {V}eronese double cones.
\newblock {\em Algebr. Geom.}, 2(2):176--192, 2015.
\newblock doi: 10.14231/AG-2015-008.

\bibitem[HM04]{HM04}
Jun-Muk Hwang and Ngaiming Mok.
\newblock Birationality of the tangent map for minimal rational curves.
\newblock {\em Asian J. Math.}, 8(1):51--63, 2004.
\newblock doi: 10.4310/AJM.2004.v8.n1.a6.

\bibitem[Hsi15]{Hsi15}
Jen-Chieh Hsiao.
\newblock A remark on bigness of the tangent bundle of a smooth projective
  variety and {$D$}-simplicity of its section rings.
\newblock {\em J. Algebra Appl.}, 14(7):1550098, 10, 2015.
\newblock doi: 10.1142/S021949881550098X.

\bibitem[Iwa18]{Iwa18}
Masataka Iwai.
\newblock Characterization of pseudo-effective vector bundles by singular
  hermitian metrics.
\newblock {\em ArXiv preprint}, 1804.02146, 2018.

\bibitem[Keb02]{Kebekus2002}
Stefan Kebekus.
\newblock Families of singular rational curves.
\newblock {\em J. Algebraic Geom.}, 11(2):245--256, 2002.
\newblock doi: 10.1090/S1056-3911-01-00308-3.

\bibitem[KM98]{KM98}
J{\'a}nos Koll{\'a}r and Shigefumi Mori.
\newblock {\em Birational geometry of algebraic varieties}, volume 134 of {\em
  Cambridge Tracts in Mathematics}.
\newblock Cambridge University Press, Cambridge, 1998.
\newblock doi: 10.1017/CBO9780511662560.

\bibitem[Kol96]{Ko96}
J{\'a}nos Koll{\'a}r.
\newblock {\em Rational curves on algebraic varieties}, volume~32 of {\em
  Ergebnisse der Mathematik und ihrer Grenzgebiete. 3. Folge. A Series of
  Modern Surveys in Mathematics}.
\newblock Springer-Verlag, Berlin, 1996.
\newblock doi: 10.1007/978-3-662-03276-3.

\bibitem[KPS18]{KuznetsovProkhorovShramov2018}
Alexander~G. Kuznetsov, Yuri~G. Prokhorov, and Constantin~A. Shramov.
\newblock Hilbert schemes of lines and conics and automorphism groups of {F}ano
  threefolds.
\newblock {\em Jpn. J. Math.}, 13(1):109--185, 2018.
\newblock doi: 10.1007/s11537-017-1714-6.

\bibitem[KST89]{KST89}
B.~\`E. Kunyavski\u{\i}, A.~N. Skorobogatov, and M.~A. Tsfasman.
\newblock del {P}ezzo surfaces of degree four.
\newblock {\em M\'{e}m. Soc. Math. France (N.S.)}, (37):113, 1989.

\bibitem[Laz04a]{Laz04a}
Robert Lazarsfeld.
\newblock {\em Positivity in algebraic geometry. {I}}, volume~48 of {\em
  Ergebnisse der Mathematik und ihrer Grenzgebiete.}
\newblock Springer-Verlag, Berlin, 2004.
\newblock doi: 10.1007/978-3-642-18808-4.

\bibitem[Laz04b]{Laz04b}
Robert Lazarsfeld.
\newblock {\em Positivity in algebraic geometry. {II}}, volume~49 of {\em
  Ergebnisse der Mathematik und ihrer Grenzgebiete.}
\newblock Springer-Verlag, Berlin, 2004.
\newblock doi: 10.1007/978-3-642-18810-7.

\bibitem[LO18]{LehmannOttem2018}
Brian {Lehmann} and John~Christian {Ottem}.
\newblock {Positivity of the diagonal.}
\newblock {\em {Adv. Math.}}, 335:664--695, 2018.
\newblock doi: 10.1016/j.aim.2018.07.002.

\bibitem[Mal20]{Mal20}
Devlin Mallory.
\newblock Bigness of the tangent bundle of del {P}ezzo surfaces and
  {D}-simplicity.
\newblock {\em arXiv preprint}, 2002.11010, 2020.

\bibitem[Mat18]{Mat18}
Shin-Ichi Matsumura.
\newblock On projective manifolds with semi-positive holomorphic sectional
  curvature.
\newblock {\em ArXiv preprint}, 1811.04182, 2018.

\bibitem[MM15]{ManivelMichalek2015}
Laurent {Manivel} and Mateusz {Micha{\l}ek}.
\newblock {Secants of minuscule and cominuscule minimal orbits.}
\newblock {\em {Linear Algebra Appl.}}, 481:288--312, 2015.
\newblock doi: 10.1016/j.laa.2015.04.027.

\bibitem[OCW16]{OCW16}
Gianluca Occhetta, Luis E.~Sol\'{a} Conde, and Kiwamu Watanabe.
\newblock Uniform families of minimal rational curves on {F}ano manifolds.
\newblock {\em Rev. Mat. Complut.}, 29(2):423--437, 2016.
\newblock doi: 10.1007/s13163-015-0183-9.

\bibitem[Pet12]{Pet12}
Thomas Peternell.
\newblock Varieties with generically nef tangent bundles.
\newblock {\em J. Eur. Math. Soc. (JEMS)}, 14(2):571--603, 2012.
\newblock doi: 10.4171/JEMS/312.

\bibitem[Sar82]{Sar82}
Victor~G. Sarkisov.
\newblock On conic bundle structures.
\newblock {\em Izv. Akad. Nauk SSSR Ser. Mat.}, 46(2):371--408, 432, 1982.

\bibitem[Sch92]{Schneider1992}
Michael Schneider.
\newblock Symmetric differential forms as embedding obstructions and vanishing
  theorems.
\newblock {\em J. Algebraic Geom.}, 1(2):175--181, 1992.

\bibitem[SCW04]{SW04}
Luis~Eduardo Sol{\'a}~Conde and Jaros{\l}aw~A. Wi{\'s}niewski.
\newblock On manifolds whose tangent bundle is big and 1-ample.
\newblock {\em Proc. London Math. Soc. (3)}, 89(2):273--290, 2004.
\newblock doi: 10.1112/S0024611504014856.

\bibitem[Sha99]{EncV}
I.~R. Shafarevich, editor.
\newblock {\em Algebraic geometry. {V}}, volume~47 of {\em Encyclopaedia of
  Mathematical Sciences}.
\newblock Springer-Verlag, Berlin, 1999.
\newblock Fano varieties, A translation of {{\i}t Algebraic geometry. 5}
  (Russian), Ross. Akad. Nauk, Vseross. Inst. Nauchn. i Tekhn. Inform., Moscow,
  Translation edited by A. N. Parshin and I. R. Shafarevich.

\bibitem[Tih80]{Tik80}
A.~S. Tihomirov.
\newblock Geometry of the {F}ano surface of a double {${\bf P}^{3}$} branched
  in a quartic.
\newblock {\em Izv. Akad. Nauk SSSR Ser. Mat.}, 44(2):415--442, 479, 1980.

\bibitem[Tik81]{Tik81}
A.~S. Tikhomirov.
\newblock The {F}ano surface of the {V}eronese double cone.
\newblock {\em Izv. Akad. Nauk SSSR Ser. Mat.}, 45(5):1121--1197, 1199, 1981.

\bibitem[Wel81]{Wel81}
Gerald.~E. Welters.
\newblock {\em Abel-{J}acobi isogenies for certain types of {F}ano threefolds},
  volume 141 of {\em Mathematical Centre Tracts}.
\newblock Mathematisch Centrum, Amsterdam, 1981.

\end{thebibliography}
\end{document}